\RequirePackage{fix-cm}
\def\version{7 October 2021}

\documentclass[smallextended]{svjour3a} %% AC's proprietary size

\smartqed  % flush right qed marks, e.g. at end of proof
\usepackage{graphicx}

\newcommand{\ac}[1]{\textcolor{red}
{{\rm [\![}\mbox{\sc{AC}$\blacktriangleright\!\!\blacktriangleright$}: #1{\rm ]\!]}}}

\usepackage{mathtools}
\usepackage{latexsym,epsfig,bm}
\usepackage{upgreek}
\usepackage{mathrsfs}
\usepackage{times}
\usepackage{amsmath}
\usepackage{amssymb}

\usepackage{amsthm}
\usepackage{upref}
\usepackage{esint}

\usepackage{comment}
\usepackage[usenames,dvipsnames]{color}
\usepackage{graphicx}
\definecolor{MyDarkBlue}{rgb}{0,0.08,0.45}
\definecolor{MyDarkGreen}{rgb}{0,0.75,0}
\definecolor{Pomegranade}{rgb}{0.6,0.1,0.15}
\definecolor{purple}{rgb}{0.6,0.1,0.15}
\usepackage[backref=none]{hyperref}
\hypersetup{pdfborder={0 0 0},
  colorlinks,
  urlcolor={MyDarkBlue},
  linkcolor={MyDarkBlue},
  citecolor={MyDarkBlue},
  breaklinks=true}
\providecommand{\url}[1]{\small\textcolor{blue}{#1}}
\usepackage{doi}
\providecommand{\eprint}[1]{}
\renewcommand{\eprint}[1]{arXiv:\href{http://arxiv.org/abs/#1}{#1}}

\usepackage{ulem} %to-be-commented-out
\providecommand{\eqref}[1]{{\rm (\ref{#1})}}
\providecommand{\itref}[1]{{\it (\ref{#1})}}
\DeclareSymbolFont{AMSb}{U}{msb}{m}{n}
\DeclareSymbolFontAlphabet{\mathbb}{AMSb}

\makeatletter\let\c@lemma\relax\makeatother

\newtheoremstyle{opusit}%                % Name
  {}%                                     % Space above
  {}%                                     % Space below
  {\it}%                             % Body font
  {}%                                     % Indent amount
  {\bfseries}%                            % Theorem head font
  {\ \ }%                                    % Punctuation after theorem head
  { }%                                    % Space after theorem head, ' ', or \newline
  {\thmname{#1}\thmnumber{ #2}\thmnote{ (#3)}}% Theorem head spec (can be left empty, meaning `normal')

\newtheoremstyle{opus}%                % Name
  {}%                                     % Space above
  {}%                                     % Space below
  {}%                             % Body font
  {}%                                     % Indent amount
  {\bfseries}%                            % Theorem head font
  {\ \ }%                                    % Punctuation after theorem head
  { }%                                    % Space after theorem head, ' ', or \newline
  {\thmname{#1}\thmnumber{ #2}\thmnote{ (#3)}}% Theorem head spec (can be left empty, meaning `normal')

\newtheoremstyle{opustext}%                % Name
  {}%                                     % Space above
  {}%                                     % Space below
  {}%                             % Body font
  {}%                                     % Indent amount
  {\it}%                            % Theorem head font
  {\ \ }%                                    % Punctuation after theorem head
  { }%                                    % Space after theorem head, ' ', or \newline
  {\thmname{#1}\thmnumber{ #2}\thmnote{ (#3)}}% Theorem head spec (can be left empty, meaning `normal')

\theoremstyle{opusit}
\newtheorem{lemma}{Lemma}[section]
\newtheorem{theorem}[lemma]{Theorem}
\newtheorem{corollary}[lemma]{Corollary}
\newtheorem{proposition}[lemma]{Proposition}
\theoremstyle{opus}
\newtheorem{definition}[lemma]{Definition}
\newtheorem{assumption}[lemma]{Assumption}
\newtheorem{example}[lemma]{Example}
\theoremstyle{opustext}
\newtheorem{remark}[lemma]{Remark}
\makeatletter\@addtoreset{definition}{section}
\makeatletter\@addtoreset{equation}{section}
\makeatletter\@addtoreset{example}{section}
\makeatletter\@addtoreset{lemma}{section}
\makeatletter\@addtoreset{proposition}{section}
\makeatletter\@addtoreset{remark}{section}
\makeatletter\@addtoreset{theorem}{section}
\makeatother

\newcommand{\unity}{\textrm{{\usefont{U}{fplmbb}{m}{n}1}}}

\newcommand{\fra}[2]{{#1}/{#2}}

\providecommand{\longhookrightarrow}{\lhook\joinrel\longrightarrow}

\DeclareSymbolFont{EUB}{U}{eur}{b}{n}
\SetSymbolFont{EUB}{bold}{U}{eur}{b}{n}
\DeclareSymbolFontAlphabet{\eub}{EUB}

\newcommand{\dom}{\mathfrak{D}}
\newcommand{\range}{\mathfrak{R}}
\renewcommand{\ker}{\mathop{\mathbf{ker}}}

\newcommand{\jj}{\mathrm{i}}
\newcommand{\End}{\,{\rm End}\,}

\newcommand{\frakM}{\mathfrak{M}}

\newcommand{\scrB}{\mathscr{B}}
\newcommand{\scrC}{\mathscr{C}}

\newcommand{\scrQ}{\mathscr{Q}}

\newcommand{\bfB}{\mathbf{B}}

\newcommand{\bfE}{\mathbf{E}}
\newcommand{\bfF}{\mathbf{F}}
\newcommand{\bfG}{\mathbf{G}}
\newcommand{\bfH}{\mathbf{H}}

\newcommand{\bfX}{\mathbf{X}}

\newcommand{\vargamma}{\mathrm{g}}
\newcommand{\hatA}{\hat{A}}

\newcommand{\notyet}[1]{{}}

\newcommand{\rank}{\mathop{\rm rank}}

\newcommand{\supp}{\mathop{\rm supp}}

\newcommand{\p}{\partial}
\newcommand{\at}[1]{\vert\sb{\sb{#1}}}

\newcommand{\R}{\mathbb{R}}
\providecommand{\C}{\mathbb{C}}
\renewcommand{\C}{\mathbb{C}}

\newcommand{\N}{\mathbb{N}}

\newcommand{\abs}[1]{\vert #1 \vert}

\newcommand{\norm}[1]{\Vert #1 \Vert}

\newcommand{\sothat}{\,{\rm :}\;}
\newcommand{\slim}{\mathop{\mbox{\rm s-lim}}}
\newcommand{\wlim}{\mathop{\mbox{\rm w-lim}}}

\DeclareMathSymbol{\varGamma}{\mathord}{letters}{"00}
\DeclareMathSymbol{\varDelta}{\mathord}{letters}{"01}
\DeclareMathSymbol{\varTheta}{\mathord}{letters}{"02}
\DeclareMathSymbol{\varLambda}{\mathord}{letters}{"03}
\DeclareMathSymbol{\varXi}{\mathord}{letters}{"04}
\DeclareMathSymbol{\varPi}{\mathord}{letters}{"05}
\DeclareMathSymbol{\varSigma}{\mathord}{letters}{"06}
\DeclareMathSymbol{\varUpsilon}{\mathord}{letters}{"07}
\DeclareMathSymbol{\varPhi}{\mathord}{letters}{"08}
\DeclareMathSymbol{\varPsi}{\mathord}{letters}{"09}
\DeclareMathSymbol{\varOmega}{\mathord}{letters}{"0A}
\newcommand{\dist}{\mathop{\rm dist}\nolimits}
\renewcommand{\Re}{\mathop{\rm{R\hskip -1pt e}}\nolimits}

\numberwithin{equation}{section}

\journalname{Annales math\'{e}matiques du Qu\'{e}bec}

\newcommand{\shortto}[1][3pt]{\mathrel{%
   \hbox{\hskip 1pt\rule[\dimexpr\fontdimen22\textfont2-.2pt\relax]{#1}{.4pt}}%
   \mkern-6mu\hbox{\usefont{U}{lasy}{m}{n}\symbol{41}}}}

\newenvironment{borrowed}
  {\list{}{\listparindent 0cm%
  \baselineskip 0pt
  \setlength{\leftmargin}{0.3cm}
  \setlength{\rightmargin}{0.2cm}
  }%
  \item\relax\it}
  {\endlist}

\makeatletter
\def\@biblabel#1{[{#1}]}
\makeatother

\begin{document}

\title{Limiting absorption principle and virtual levels
of operators in Banach spaces%\thanks{Grants or other notes
%about the article that should go on the front page should be
%placed here. General acknowledgments should be placed at the end of the article.}
}
%\subtitle{Do you have a subtitle?\\ If so, write it here}

%\titlerunning{Short form of title}        % if too long for running head

\author{Nabile Boussaid         \and
        Andrew Comech
}

%\authorrunning{Short form of author list} % if too long for running head

\institute{N. Boussaid \at
Laboratoire Math\'{e}matiques, %% de Besan\c{c}on,
Universit\'e Bourgogne Franche-Comt\'e,
25030 Besan\c{c}on CEDEX, France \\
%              Tel.: +123-45-678910\\
%              Fax: +123-45-678910\\
              \email{nabile.boussaid@univ-fcomte.fr}           %  \\
%             \emph{Present address:} of F. Author  %  if needed
           \and
           A. Comech \at
Texas A\&M University, College Station, Texas, USA;
Laboratory 4, IITP, Moscow, Russia \\
              \email{comech@math.tamu.edu}
}

%\date{Received: date / Accepted: date}
\date{\version}
% The correct dates will be entered by the editor

\maketitle

\begin{abstract}
We review the concept of the limiting absorption principle
and its connection to virtual levels of operators in Banach spaces.

\medskip

\noindent\textsc{R\'esum\'e.}
Nous passons en revue le principe d'absorption limite
et sa relation avec les niveaux virtuels pour des op\'erateurs dans les espaces de Banach.

\end{abstract}

\bigskip
\bigskip

\hfill
{\it To Alexander Shnirelman on the occasion of his 75th birthday}

%\bigskip

\section{Limiting absorption principle}

The idea of introducing
a small absorption into the wave equation
for selecting particular solutions
goes back to Ignatowsky \cite{ignatowsky1905reflexion}
and is closely related to the Sommerfeld radiation condition
\cite{sommerfeld1912greensche}.
We start with the Helmholtz equation
\begin{eqnarray}\label{264-1}
-\Delta u-z u=f(x)\in L^2(\R^3),
\qquad u=u(x),\quad x\in\R^3.
\end{eqnarray}
For $z\in\C\setminus\overline{\R_{+}}$,
equation
\eqref{264-1} has a unique $L^2$-solution
%given by
$(-\Delta-z I)^{-1}f$,
with $(-\Delta-z I)^{-1}$
represented by the convolution with
$\fra{e^{-\abs{x}\sqrt{-z}}}{(4\pi\abs{x})}$,
$\Re\sqrt{-z}>0$.
For $z\ge 0$, there may be no
$L^2$-solution;
moreover, when $z>0$, there are different solutions
of similar norm, and one faces the question of
choosing an appropriate one.
The way to select a unique solution
is known as the \emph{radiation principle}.
%% In a widely renowned ``Course of higher mathematics'',
%% in volume four \cite{smirnov1941course},
%% V.I.\,Smir\-nov credits the radiation principle
V.I.\,Smir\-nov,
in his widely renowned ``Course of higher mathematics''
%in volume four 
\cite{smirnov1941course},
credits the radiation principle
to V.S.\,Ig\-na\-towsky \cite{ignatowsky1905reflexion}
and to A. Sommerfeld \cite{sommerfeld1912greensche};
the work of Ignatowsky
was also publicized by A.N.\,Tikhonov,
both in his lectures at mechmat at Moscow State University
and in the textbook written jointly with A.A.\,Samarskii \cite{tikhonov1951equations}
(and even in their '1950 Russian translation
of A.\,Sommerfeld's textbook \cite{sommerfeld1948partielle}).
%%also mentioning \cite{sternberg1936anwendung} (who also cites Ignatowsky).
%This article is mentioned by
%Sommerfeld (in an unrelated way) in \cite{sommerfeld1909ausbreitung}
In \cite{ignatowsky1905reflexion},
Ignatowsky considered the electromagnetic field
scattered by a wire
using the expression $Z(t,x)=A e^{\jj(\omega t-\varkappa x)}$
for the electric field,
with $\omega$ and $\varkappa$ certain parameters.
The absorption in the medium corresponded to
$\varkappa$ having nonzero imaginary part;
its sign was taken into account
when choosing an appropriate solution
to the Helmholtz equation.
Following this idea,
A.G.\,Sveshnikov, a student of Tikhonov,
specifies in \cite{sveshnikov1950radiation} a solution to \eqref{264-1} by
\begin{eqnarray}\label{lap}
u(x)=\lim\sb{\varepsilon\to 0+}
\big(-\Delta-(z+\jj\varepsilon)I\big)^{-1}f,
\qquad
k>0,
\end{eqnarray}
calling this approach \emph{the limiting absorption principle} (LAP)
and attributing it to Ignatowsky.\footnote{We suppose that
in the twenties and thirties,
between \cite{ignatowsky1905reflexion}
and \cite{smirnov1941course},
the idea of the \emph{limiting absorption principle}
was being refined
when
V.S.\,Ignatowsky worked at St.\,Petersburg University,
where in particular he taught mathematical theory of diffraction
and likely was in contact with V.I.\,Smirnov.
Let us mention that, besides his work on diffraction,
Ignatowsky is known
for his contributions to the theory of relativity
(see \cite{vizgin1987reception})
and for developing optical devices
while heading the theoretical division
at GOMZ, the State Association
for Optics and Mechanics (which later became known as ``LOMO'').
On 6 November 1941, during the blockade of St.\,Petersburg,
Ignatowsky was arrested by NKVD
(an earlier name of KGB),
as a part of the ``process of scientists'',
and shot on 30 January 1942.
(During this process,
V.I.\,Smirnov was ``credited'' by NKVD the role of a Prime Minister
in the government after the purportedly planned
%anti-Stalin
coup;
Smirnov avoided the arrest because he was evacuated from St.\,Petersburg
in August 1941, shortly before the blockade began.)
As a result, Ignatowsky's name has been unknown to many:
%in the after-war edition of Smirnov's course \cite{smirnov1951course},
the reference to his article disappeared
from Smirnov's ``Course of higher mathematics''
until post-1953 editions
(see e.g. the English translation \cite[\S230]{smirnov1964course}).

\qquad
Russians are used to such
rewrites of the history,
joking about the ``History of the history of the Communist Party'',
a reference to a mandatory and ever-changing Soviet-era ideological course
in the first year of college.
As the matter of fact, the very ``Course of higher mathematics''
mentioned above
was started by V.I. Smirnov together with J.D. Tamarkin,
with the first two volumes
%%\cite{smirnov1924course,smirnov1926course}
(published in 1924 and 1926)
bearing both names,
% \cite{smirnov1924course,smirnov1926course},
but after Tamarkin's persecution by GPU
(another earlier name of KGB)
and his escape from the Soviet Union with smugglers
over frozen lake Chudskoe in December 1924 \cite{hille1947jacob},
Tamarkin's authorship eventually had to disappear.
His coauthor Smirnov
spent the next year pleading
with the authorities
(and succeeding!)
for Tamarkin's wife Helene Weichardt
-- who tried to follow her husband's route
with the smugglers over the icy lake but was intercepted
at the border and jailed --
to be released from prison
and allowed to leave the Soviet Union to join her husband
\cite{apushkinskaya2018cherish}.
}
%%
%% https://www.keldysh.ru/memory/tikhonov/book/topic_21.htm
%%
We note that \eqref{lap} leads to
\begin{eqnarray}\label{u-x}
u(x)\mathop{\sim}\limits\sb{r\to\infty}
\ \lim\sb{\varepsilon\to +0}
e^{\jj r\sqrt{z+\jj\varepsilon}}
=e^{\jj k r},
\qquad
k=z^{1/2}>0,
\qquad
r=\abs{x},
\end{eqnarray}
%\nb{shall we add $r\sim \infty$ to $\sim$ ?}\ac{sure}
where the choice of the branch of the square root in the exponent
is dictated by the need to avoid the exponential growth at infinity.
%%one has to choose the ``$+$''-sign in the exponent.
Sveshnikov points out that Ignatowsky's approach
does not depend on the geometry of the domain
and hence is of more universal nature
than that of A. Sommerfeld \cite{sommerfeld1912greensche},
which is the selection of the solution to \eqref{264-1}
satisfying the \emph{Sommerfeld radiation condition}
%%\cite[Eq.\,(21)]{sommerfeld1912greensche}
\begin{eqnarray}\label{Sommerfeld}
\lim\sb{r\to\infty}
r\Big(
\frac{\p u}{\p r}-\jj k u
\Big)=0,
\end{eqnarray}
in agreement with \eqref{u-x}.

Let us also mention
the \emph{limiting amplitude principle} \cite{tikhonov1948radiation}
(the terminology also introduced in \cite{sveshnikov1950radiation})
which specifies a solution to \eqref{264-1}
by
$
u(x)=\!\!\lim\limits\sb{t\to+\infty}\!\!\psi(x,t)e^{\jj k t}$,
where $\psi(x,t)$ is a solution to the wave equation
\begin{eqnarray}\label{264-1a}
\p_t^2\psi-\Delta \psi=f(x)e^{-\jj k t},
\qquad
t>0;
\qquad
\psi\at{t=0}=0,
\qquad
\p_t \psi\at{t=0}=0.
\end{eqnarray}
%%satisfying the zero initial condition.
Thus, $u$ is the limiting amplitude of the periodic vibration
building up
under the action of a periodic force for a long time.
This corresponds to using the retarded Green function,
represented by the convolution with
$G_{\mathrm{ret}}(x,t)=\frac{\delta(t-\abs{x})}{4\pi\abs{x}}$,
yielding the solution to \eqref{264-1a} in the form
\[
%%begin{eqnarray}\label{retarded}
\psi(x,t)=G_{\mathrm{ret}}\ast (f(x) e^{-\jj k t})
%% =
%% \iiint\int
%% \frac{
%% \delta(t-s-\abs{x-y})
%% f(y)e^{-\jj k s}}{4\pi\abs{x-y}}
%% \,dy\,ds
=
\iiint_{\abs{x-y}<t}
\frac{f(y)e^{-\jj k (t-\abs{x-y})}}{4\pi\abs{x-y}}\,dy
\sim e^{\jj k\abs{x}-\jj k t}
,
%%\end{eqnarray}
\]
% \nb{why is th integral limited to $\abs{x-y}<t$ ?}
% \ac{Here is the intermediate term,
% $G_{\mathrm{ret}}\ast (f(x) e^{-\jj k t})
% =
% \iiint\int
% \frac{
% \delta(t-s-\abs{x-y})
% f(y)e^{-\jj k s}}{4\pi\abs{x-y}}
% \,dy\,ds
% =...
% $
% }
in agreement both
with
the limiting absorption principle~\eqref{lap} (cf.~\eqref{u-x})
and with the Sommerfeld radiation condition~\eqref{Sommerfeld}.

Presently, a common meaning of the LAP
is the existence of a limit of the resolvent
at a given point of the essential spectrum.
While the resolvent of $A:\,\bfX\to\bfX$
cannot have a limit at the essential spectrum
as an operator in $\bfX$,
it can have a limit as a mapping
\[
(A-z I)^{-1}:\,\bfE\to\bfF,
\]
where $\bfE$ and $\bfF$ are some Banach spaces
such that the embeddings
$\bfE\hookrightarrow\bfX\hookrightarrow\bfF$
are dense and continuous.
Historically,
this idea could be traced back to
eigenfunction expansions
\cite{weyl1910gewohnliche,carleman1934theorie,titchmarsh1946eigenfunction}
and
Krein's method of \emph{directing functionals}
\cite{krein1946general,krein1948hermitian}
(see \cite[Appendix II.7]{akhiezer1981theory}).
This was developed in
%% into Gelfand--Kostyuchenko theory
\cite{povzner1950method,povzner1953expansion,gelfand1955eigenfunction,berezanskii1957eigenfunction,birman1961spectrum}
(see also \emph{rigged spaces} in \cite[I\S4]{gelfand1961some},
also known as \emph{equipped spaces}
and related to \emph{Gelfand's triples} from \cite{gelfand1955eigenfunction}).
%A related idea
%of completing the underlying Hilbert space
%in the metric induced by the quadratic form
%defined by a positive-definite selfadjoint operator 
%is developed in \cite{birman1961spectrum}.
Gradually the theory takes the form
of estimates on the limit of the resolvent
at the essential spectrum
in certain spaces;
this further development becomes clearer in
\cite{eidus1962principle,vainberg1966principles,eidus1969principle}
(the convergence of the resolvent
is in the sense of distributions),
then in
%%\cite{rejto1967partly},
\cite[Lemma 6.1]{rejto1969partly}
(where certain spaces are introduced),
and finally in
\cite[Theorem 2.2]{agmon1970spectral},
\cite[Theorem 4.1]{MR0320547} (for Dirac operators),
and \cite[Appendix A]{agmon1975spectral},
where the convergence of the resolvent is specified
with respect to weighted $L^2$ spaces.
See also \cite{MR528757}
and
%For the relation between LAP and the smoothness properties
%of the spectral measure, see
\cite{benartzi1987limiting}.
Let us also mention that
in \cite{agmon1998perturbation}
this same approach --
to consider the resolvent as a mapping from $\bfE$ to $\bfF$,
with the embeddings $\bfE\hookrightarrow\bfX\hookrightarrow\bfF$
being dense and continuous -- is used to define \emph{resonances}
of an operator as poles of the analytic continuation
of its resolvent.

\begin{remark}
\label{remark-roma}
%\ac{\sout{We know from Roman V. Romanov}}
%\nb{I suggest to postpone to aknowledgments the mentions to R. Romanov.}
%% that
%\ac{OK}
Such an approach is not universal
since such a definition of resonances depends on the
choice of regularizing spaces $\bfE$, $\bfF$.
By \cite[Proposition 4.1]{agmon1998perturbation},
the set of resonances is the same if $\bfE_i$ and $\bfF_j$,
$i=1,\,2,$ satisfy the following additional assumptions:
\begin{itemize}
\item[(I)]
{\it
The set $\bfE_1\cap\bfE_2$ (identified with a subset of $\bfX$)
is dense in both $\bfE_1$ and $\bfE_2$;}
\item[(II)]
{\it
There exists a Banach space $\bfF$ containing both $\bfF_1$ and $\bfF_2$
as linear subsets with embeddings which are continuous.}
\end{itemize}
See Example~\ref{example-1d}
and Theorem~\ref{theorem-a} below.
\begin{comment}
Let us point out that the example from \cite{howland1974poiseux},
\begin{borrowed}
If $H$ is the multiplication
$H f(x)=x f(x)$, $-\infty<x<\infty$,
then given any point $\varLambda$ in the lower half-plane,
there is a rational function $f(x)$ for which the continuation of
$((H-\zeta)^{-1}f,f)=\int (x-\zeta)^{-1}|f(x)|^2\,dx$
has a point at $\varLambda$.
\end{borrowed}
\noindent
seems somewhat artificial since
the extra resonances are introduced from the definition of $\bfE$ and $\bfF$.
\end{comment}
\end{remark}
%% through the essential spectrum.
%% a point $z_1$ located on the unphysical sheet of the Riemannian
%% surface of the spectral parameter is called a resonance if
%% the mapping
%% $(A-z I)^{-1}:\,\bfE\to\bfF$,
%% for any choice of $\bfE$ and $\bfF$,
%% cannot be extended
%% analytically to an open neighborhood of $z_1$.

%%\cite{kato1970}.
%The concept of \emph{partly gentle perturbations}
%\cite{rejto1967partly}

Perhaps the simplest example of LAP
is covered by S.\,Agmon in \cite[Lemma A.1]{agmon1975spectral}:
by that lemma,
the operator $(\p_x-z I)^{-1}$, $z\in\C$, $\Re z\ne 0$,
is uniformly bounded as an operator from $L^2_s(\R)$ to $L^2_{-s}(\R)$, $s>1/2$,
and has a limit (in the uniform operator topology)
as $\Re z\to \pm 0$. For example, for $\Re z<0$,
the solution to the equation $(\p_x-z)u=f$ is given by the operator
$f\mapsto u(x)=\int_{-\infty}^x e^{z(x-y)}f(y)\,dy$,
which is bounded from $L^1(\R)$ to $L^\infty(\R)$
%% (having a limit in the strong operator topology as $z\to z_0\in\jj\R$)
and hence from $L^2_s(\R)$ to $L^2_{-s}(\R)$, $s>1/2$,
uniformly in $z\in\C$, $\Re z<0$.
%% for $z_0\in\jj\R$, the limit of this operator as $z\to z_0-0$ is given by
%% $f\mapsto u(x)=\int_{-\infty}^x e^{z_0(x-y)}f(y)\,dy$.
Here we use the standard notation
\begin{eqnarray}\label{def-lp}
L^p_s(\R^d)=\{
u\in L^p_{\mathrm{loc}}(\R^d);
\,
\langle \cdot\rangle^s u\in L^p(\R^d),
\,
\norm{u}_{L^p_s}
=\norm{\langle\cdot\rangle^s u}_{L^p}
\},
\end{eqnarray}
for any $p\in [1,+\infty]$,
$s\in\R$, $d\in\N$,
%%with $\langle\cdot \rangle$ the operator of multiplication by
with $\langle x\rangle=(1+x^2)^{1/2}$.
Agmon then shows that the LAP is available
for the Laplacian
when the spectral parameter approaches the bulk of the
essential spectrum:
by \cite[Theorem 4.1]{agmon1975spectral},
for $s,\,s'>1/2$,
the resolvent
\begin{eqnarray}\label{r0-ls}
R_0^{(d)}(z)=
(-\Delta-z I)^{-1}:\;L^2_{s}(\R^d)\to L^2_{-s'}(\R^d),
\quad
z\in\C\setminus\overline{\R_{+}},
\quad
d\ge 1,
\quad
\end{eqnarray}
is bounded uniformly for $z\in\varOmega\setminus\overline{\R_{+}}$,
for any open neighborhood $\varOmega\subset\C$
such that
$\overline\varOmega\not\ni \{0\}$,
and has limits as $z\to z_0\pm \jj 0$, $z_0>0$.
For the sharp version
(the $\bfB\to\bfB^*$ continuity of the resolvent
in the Agmon--H\"{o}rmander spaces), see
\cite[Proposition 6.3.6]{MR2598115}.

While the mapping \eqref{r0-ls}
has a limit as $z\to z_0\pm \jj 0$ with $z_0>0$, for any $d\ge 1$,
the behaviour at $z_0=0$ depends on $d$.
For example, in three dimensions,
as long as $s,\,s'>1/2$ and $s+s'\ge 2$,
the mapping \eqref{r0-ls},
represented by the convolution with
$(4\pi\abs{x})^{-1}e^{-\abs{x}\sqrt{-z}}$,
$\ z\in\C\setminus\overline{\R_{+}}$,
$\ \Re\sqrt{-z}>0$,
remains uniformly bounded
and has a limit as $z\to z_0=0$.
A similar boundedness of the resolvent in an open neighborhood of the threshold
$z_0=0$ persists in higher dimensions,
but breaks down in dimensions $d\le 2$.
In particular, for $d=1$,
the resolvent is represented by the convolution with
$e^{-\abs{x}\sqrt{-z}}/(2\sqrt{-z})$,
$z\in\C\setminus\overline{\R_{+}}$,
$\Re\sqrt{-z}>0$,
and cannot have a limit as $z\to 0$
as a mapping $\bfE\to\bfF$
as long as $\bfE,\,\bfF$ are weighted Lebesgue spaces
(at the same time, see Example~\ref{example-1d} below).
There is a similar situation in two dimensions.
We say that the threshold
$z_0=0$ is a \emph{regular point} of the essential spectrum
for $d\ge 3$ and that it is a \emph{virtual level}
if $d\le 2$.

\begin{example}
\label{example-1d}
While
the limit of the resolvent
$(-\p_x^2-z I)^{-1}$, $z\to z_0=0$,
does not exist
in the weak operator topology
of mappings $L^2_s(\R)\to L^2_{-s}(\R)$ with
arbitrarily large $s>1$,
%\ac{\sout{we learned from Roman V. Romanov that
%%the limit of $(-\p_x^2-z I)^{-1}$, $z\to z_0=0$,}}
this limit exists
in the weak operator topology of mappings
$\bfE\to\bfF$
if one takes
\[
\bfE=\big\{u\in L^2_s(\R)\sothat
\mbox{$\hat u(\xi)$ vanishes of order $\tau>1$ at $\xi=0$}
\big\},
\ \ \bfF=L^2_{-s}(\R),
\ \ s>1,
\]
with $\norm{u}_\bfE:=\norm{u}_{L^2_s}
+\mathop{\lim\sup}_{\xi\to 0}\abs{\xi}^{-\tau}
\abs{\hat u(\xi)}$. 
% \nb{$\abs{\xi}^{-\tau}$, no ?}
%\ac{sure!}
Both $L^2_s(\R)$ and $\bfE$ 
are densely and continuously embedded into $\bfX=L^2(\R)$,
while $\bfE\cap L^2_s(\R)=\bfE$ is not dense in $L^2_s(\R)$
(cf. Remark~\ref{remark-roma}):
for a fixed $v\in L^2_s(\R^d)$ with $\hat v(0)\ne 0$
and for any $u\in\bfE$,
one has
\[
\norm{u-v}_{L^2_s}
=\norm{\hat u-\hat v}_{H^s(\R^d)}
\ge c_s\abs{\hat u(0)-\hat v(0)}
=c_s\abs{\hat v(0)},
\]
where $c_s>0$ depends only on $s>d/2$;
thus the left-hand side
cannot approach zero.
%\ac{\sout{See \cite{virtual-levels} for more details.}}
\end{example}

%% In three dimensions,
%% and also takes place for Schr\"odinger operators
%% $H=-\Delta+V(x)$, with $V$ a real-valued function
%% with appropriate decay at infinity,
%% as long as $z_0$ is a \emph{regular point} of the essential
%% spectrum (in the sense to be defined below).
%% Such a boundedness is absent
%% for the resolvent of the free Laplace operator
%% in dimensions $d\le 2$
%% because of the presence of a virtual level at $z_0=0$,
%% but becomes available if some
%% perturbation $V$ is added.

\section{Virtual levels}
\label{sect-virtual}

\noindent
{\bf History of virtual levels.\ }
Virtual levels appeared first in the nuclear physics,
in the study of neutron scattering on protons by
% Eugene Wigner
E.\,Wigner
\cite{wigner1933streuung}.
%% just a year after the discovery of the neutron by James Chadwick.
While a proton and a neutron with parallel spins
form a spin-one deuteron
(Deuterium's nucleus),
%a so-called triplet state $^3\! S$ (spin one),
which is stable,
with the binding energy around $2.2$\,MeV,
when the spins
of the particles
are antiparallel,
their binding energy
is near zero.
It was not clear
for some time
whether the corresponding
%% singlet state $^1\! S$ (spin zero)
spin-zero state
is \emph{real} or \emph{virtual},
that is, whether the binding energy was positive or negative;
see, for instance, \cite{fermi1935recombination},
where the word ``virtual'' appears first.
It turned out that this state was \emph{virtual} indeed \cite{PhysRev.50.899},
with a small \emph{negative} binding energy, around $-67$\,KeV.
%% The presence of this virtual state
%% manifests itself in the increase of the cross-section
%% at slow energies,
%% when slow incoming neutrons try to couple with protons.
%``virtual''
% singlet
%state of the deuteron
%of energy $E\approx 0$.
The resulting increase in the total cross-section
of the neutron scattering on protons
is interpreted as a resonance
of the incoming wave with this ``virtual state''
corresponding to the energy $E\approx 0$.

Mathematically, virtual levels
%as well as embedded and threshold eigenvalues,
correspond to particular singularities
of the resolvent at the essential spectrum.
This idea goes back to 
J.\,Schwinger \cite{schwinger1960field}
and was further addressed by
M.\,Birman \cite{birman1961spectrum},
L.\,Faddeev \cite{MR0163695},
%% Faddeev \cite{MR0163695,MR0178362},
B.\,Simon \cite{MR0353896,simon1976bound},
B.\,Vainberg \cite{vainberg1968analytical,vainberg1975short},
D.\,Yafaev \cite{yafaev1974theory,yafaev1975virtual},
J.\,Rauch \cite{MR0495958},
and A.\,Jensen and T.\,Kato \cite{MR544248},
%% and Vugalter and Zhislin  \cite{vugalter1982discrete},
%% \cite{zhislin1986virtual},
with the focus on
Schr\"odinger operators
in three
%$d=3$
dimensions.
%% Virtual levels for general exterior elliptic problems
%% (of any order and any dimension)
%% with constant at infinity coefficients
%% were studied by Vainberg in
%% \cite{vainberg1968analytical,vainberg1975short}.
%%(mat. sb,68, YMN,75)
Higher dimensions were considered
in \cite{MR563367,yafaev1983scattering,MR748579}.
An approach to more general
symmetric differential operators was developed in
\cite{weidl1999remarks}.
The virtual levels of nonselfadjoint
Schr\"o\-din\-ger operators in three dimensions
appeared in \cite{cuccagna2005bifurcations}.
Dimensions $d\le 2$
require special attention since the free Laplace operator
has a virtual level at zero (see \cite{simon1976bound}).
The one-dimensional case is covered in \cite{MR768299,MR877834}.
The approach from the latter article
was further developed
in \cite{MR952661}
to two dimensions
(if $\int\sb{\R^2} V(x)\,dx\ne 0$)
and then
in \cite{MR1841744} (with this condition dropped)
who give a general approach
in all dimensions,
with the regularity of the resolvent
formulated
via the weights which are square roots of the potential
(and consequently not optimal).
There is an interest in the subject due to
dependence of dispersive estimates on the presence
of virtual levels at the threshold point,
see e.g. \cite{MR544248,yafaev1983scattering,erdogan2004dispersive,yajima2005dispersive}
in the context of Schr\"{o}dinger operators;
the Dirac operators are treated in
\cite{boussaid2006stable,boussaid2008asymptotic,erdogan2017dirac,erdogan2019dispersive}.
Let us mention the dichotomy between a virtual level
%(null-critical case)
and an eigenvalue
% (positive-critical case)
manifested in the large-time behavior
of the heat kernel and the behavior of the Green function
near criticality; see \cite{pinchover1992large,pinchover2004large}.
We also mention recent articles
\cite{barth2020efimov}
on properties of virtual states
of selfadjoint Schr\"odinger operators
% in dimensions $d\le 2$
and
\cite{gesztesy2020absence}
proving the absence
of \emph{genuine} (non-$L^2$) virtual states
of selfadjoint Schr\"odinger operators
and massive and massless Dirac operators,
% in dimension $d\ge 5$
%and of massless Dirac operators in dimension  $d\ge 3$,
as well as giving classification of virtual levels and
deriving properties of eigenstates and virtual states.
%% \cite[Theorem 3.3]{gesztesy2020absence}.

\smallskip
\noindent
{\bf Equivalent characterizations of virtual levels.\ }
The definition of virtual levels has been somewhat empirical;
one would say that there were a virtual level at the
threshold of the Schr\"odinger operator
if a certain arbitrarily small perturbation could produce a (negative)
eigenvalue.
To develop a general approach for nonselfadjoint operators,
we notice
that the following properties of the threshold $z_0=0$
of the Schr\"o\-din\-ger operator $H=-\Delta+V(x)$,
$x\in\R^d$, $d\ge 1$,
$V\in C_{\mathrm{comp}}(\R^d,\C)$,
are related:

\smallskip

\noindent
{\it
(P1)\;
There is a nonzero solution to
$H\psi=z_0\psi$
from $L^2$
or a certain larger space;

\noindent
(P2)\;
$R(z)=(H-z I)^{-1}$
has no limit
in weighted spaces as $z\to z_0$;

\noindent
(P3)\;
Under an arbitrarily small perturbation,
an eigenvalue can bifurcate from $z_0$.
}%\it
\smallskip

\noindent
For example, properties~{\it (P1) -- (P3)}
are satisfied for $H=-\p_x^2$
in $L^2(\R)$
considered with domain $\dom(H)=H^2(\R)$.
%near the point $z_0=0$.
Indeed, the equation $-\p_x^2\psi=0$
has a bounded solution
$\psi(x)=1$;
while non-$L^2$,
it is ``not as bad as a generic solution''
to $(-\p_x^2+V(x))\psi=0$ with $V\in C_{\mathrm{comp}}(\R)$,
%%(which we allow to be complex-valued),
which
%\nb{may grow ?}\ac{OK!}
may grow linearly at infinity.
The integral kernel of the resolvent
$R_0^{(1)}(z)=(-\p_x^2-z I)^{-1}$, $z\in\C\setminus\overline{\R_{+}}$,
contains a singularity at $z=0$:
\begin{eqnarray}\label{free-resolvent-1d}
R_0^{(1)}(x,y;z)=
\frac{e^{-\abs{x-y}\sqrt{-z}}}{2\sqrt{-z}},
\quad
x,\,y\in\R,
\quad
z\in\C\setminus\overline{\R_{+}},
\quad
\Re\sqrt{-z}>0,
\end{eqnarray}
and has no limit as $z\to 0$
even in weighted spaces.
Under a small perturbation,
an eigenvalue may bifurcate from the
threshold (see e.g. \cite{simon1976bound}).
Indeed, for the perturbed operator
$H_\vargamma=-\p_x^2-\vargamma\unity_{[-1,1]}$, $0<\vargamma\ll 1$,
%has an eigenvalue
%% $E_\vargamma=-\kappa^2\in(-\vargamma,0)$
%$E_\vargamma=-\vargamma^2+o(\vargamma^2)\in(-\vargamma,0)$.
there is a relation
\[
%H\psi=
(-\p_x^2-\vargamma\unity_{[-1,1]})\psi(x)=-\kappa^2\psi(x),
\quad
\psi(x)=
\begin{cases}
c_1 e^{-\kappa\abs{x}},&\abs{x}>1,
\\
c_2\cos\big(\,x\sqrt{\vargamma-\kappa^2}\,\big),&\abs{x}\le 1,
\end{cases}
\]
%\nb{Do we need $\kappa^2 < g$ ?}
%\ac{OK:}
where we assume that $\kappa\in(0,\vargamma^{1/2})$.
The eigenvalue $E_\vargamma:=-\kappa^2$
%by deriving
%$\kappa$
is obtained from the continuity of $-\p_x\psi/\psi$ at $x=1\pm 0$:
\[
\kappa
=\sqrt{\vargamma-\kappa^2}
\tan\sqrt{\vargamma-\kappa^2}
=\vargamma-\kappa^2+O((\vargamma-\kappa^2)^2),
%=\vargamma+O(\vargamma^2),
\]
hence
$\kappa=\vargamma+O(\vargamma^2)$,
leading to
$E_\vargamma=-\kappa^2=-\vargamma^2+O(\vargamma^3)$.
In this case,
when properties~{\it (P1) -- (P3)}
are satisfied,
one says that $z_0=0$
is a \emph{virtual level};
%(also known as a \emph{zero-energy} or \emph{threshold resonance});
the corresponding
% non-$L^2$
nontrivial bounded solution $\psi(x)=1$
of $-\p_x^2\psi=0$
is a \emph{virtual state}.
%To emphasize that a particular virtual level
%is not an $L^2$-eigenvalue,
%we will sometimes say that this virtual level
%is \emph{genuine}, non-$L^2$.

On the contrary,
properties~{\it (P1) -- (P3)}
are not satisfied for $H=-\Delta$
in $L^2(\R^3)$,
with $\dom(H)=H^2(\R^3)$.
%near the threshold $z_0=0$.
Regarding~{\it (P1)},
we notice that
nonzero solutions to $(-\Delta+V)\psi=0$
(with certain compactly supported potentials)
can behave like the Green function,
$\sim \abs{x}^{-1}$ as $\abs{x}\to\infty$,
and one expects that this is what virtual states should look like,
while nonzero solutions to $\Delta\psi=0$
cannot have uniform decay as $\abs{x}\to\infty$,
so should not qualify as virtual states;
the integral kernel of
$R_0^{(3)}(z)=(-\Delta -z I)^{-1}$,
%% given by
\begin{eqnarray}\label{free-resolvent-3d}
R_0^{(3)}(x,y;z)=
\frac{e^{-\sqrt{-z}\abs{x-y}}}{4\pi\abs{x-y}},
\quad
x,\,y\in\R^3,
\quad
z\in\C\setminus\overline{\R_{+}},\quad\Re\sqrt{-z}>0,
\quad
\end{eqnarray}
remains pointwise bounded as $z\to 0$
and has a limit in the space of mappings
$L^2_s(\R^3)\to L^2_{-s'}(\R^3)$, $s,\,s'>1/2$, $s+s'>2$
(see e.g. \cite{MR544248}),
failing~{\it (P2)};
finally, small perturbations
cannot produce negative eigenvalues
%%(this follows from the Birman--Schwinger theory \cite[Theorem XIII.10]{MR0493421}),
(this follows from the Hardy inequality),
so~{\it (P3)} also fails.
In this case, we say that
$z_0=0$ is a \emph{regular} point of the essential spectrum.

We claim that the properties
{\it (P1) -- (P3)} are essentially equivalent,
even in the context of the general theory \cite{virtual-levels}.
These properties are satisfied
when $z_0$ is either an eigenvalue of $H$
or, more generally, a \emph{virtual level}.
To motivate the general theory,
we can start from the Laplace operator in one dimension,
considering the problem
\begin{eqnarray}\label{p-u-f}
(-\p_x^2-z)u(x)=f(x),
\qquad
u(x)\in\C,
\quad
x\in\R.
\end{eqnarray}
For any $f\in C_{\mathrm{comp}}(\R)$,
there is a $C^2$-solution to \eqref{p-u-f}.
If we consider $z\in\C\setminus\overline{\R_{+}}$,
then the natural choice of a solution is
\[
u(x)=(R_0^{(1)}(z)f)(x):=\int_{\R} R_0^{(1)}(x,y;z)f(y)\,dy,
\]
where
the resolvent $R_0^{(1)}(z)=(-\Delta-z I)^{-1}$ has the integral kernel
$R_0^{(1)}(x,y;z)$ from \eqref{free-resolvent-1d}.
%% \[
%% G(x,y;z)=\frac{e^{-\abs{x-y}\sqrt{-z}}}{2\sqrt{-z}},
%% \quad
%% x,\,y\in\R,
%% \quad
%% z\in\C\setminus\overline{\R_{+}},
%% \quad
%% \Re\sqrt{-z}>0.
%% \]
This integral kernel is built of
solutions $e^{\pm x\sqrt{-z}}$;
the choice of such a combination
is dictated by the desire to avoid solutions exponentially
growing at infinity.
For $z\ne 0$,
since $R_0^{(1)}(x,y;z)$ is bounded,
the mapping $f\mapsto R_0^{(1)} f$
defines a bounded mapping
$L^1(\R)\to L^\infty(\R)$.
This breaks down at $z=0$,
since
$e^{\pm x\sqrt{-z}}$ are linearly dependent when $z=0$.
To solve \eqref{p-u-f} at $z=0$, one can use
the convolution with the fundamental solution
$%%R_0^{(1)}(x,y;z)\at{z=0}=
G(x)=\abs{x}/2+x C$, with any $C\in\C$.
While such fundamental solutions
% (with any choice of $C$)
provide a solution
$u=G\ast f$
to \eqref{p-u-f},
this solution may no longer be from $L^\infty$;
any of the above choices of $G$
would no longer be bounded as a mapping $L^1\to L^\infty$.
This problem is resolved if
a potential $V\in C_{\mathrm{comp}}(\R,\C)$
%(we allow it to be complex-valued)
is introduced into \eqref{p-u-f},
\begin{eqnarray}\label{p-v-u-f}
(-\p_x^2+V-z_0)u=f,
\qquad
x\in\R,
\end{eqnarray}
so that the \emph{Jost solution}
$\theta_{-}(x)$
to $(-\p_x^2+V)u=0$
with $\lim\sb{x\to-\infty}\theta_{-}(x)=1$,
tends to infinity
as $x\to+\infty$ and is linearly independent with
the Jost solution $\theta_{+}(x)$,
$\lim\sb{x\to +\infty}\theta_{+}(x)=1$.
To construct a fundamental solution to \eqref{p-v-u-f}
at $z_0=0$,
we set
%% can take solutions $\theta\sb\pm(x)$
%% such that $\theta_{-}(x)\to 1$ as $x\to-\infty$,
%% $\theta_{+}(x)\to 1$ as $x\to +\infty$,
%% and
%defines
\begin{eqnarray}\label{gg}
G(x,y;z_0)
=
\frac{1}{W[\theta_{+},\theta_{-}](y)}
\begin{cases}
\theta_{-}(y)\theta_{+}(x),&x>y,
\\
\theta_{-}(x)\theta_{+}(y),&x<y,
\end{cases}
\end{eqnarray}
with $W[\theta_{+},\theta_{-}](y)
=\theta_{+}(y)\theta_{-}'(y)-\theta_{+}'(y)\theta_{-}(y)$, the Wronskian.
This will work if $\abs{\theta_{-}(x)}$ grows as $x\to+\infty$
(and similarly if $\abs{\theta_{+}(x)}$ grows as $x\to-\infty$);
if, on the other hand, $\theta\sb\pm$ remain bounded, then,
as the matter of fact, these functions are linearly dependent,
their Wronskian is zero, and \eqref{gg} is not defined.
In this construction
the space $L^\infty$ appears twice:
it contains the range of $G(z_0)\at{L^2_s(\R)}$, $s>3/2$,
when $\theta\sb\pm$ are linearly independent
(see \cite{virtual-levels}), and it is the space where
$\theta\sb\pm$ live when they are linearly dependent.
This is not a coincidence: from $-u''=f\in C^\infty_{\mathrm{comp}}(\R)$,
we can write $-u''+V u=f+V u$,
and then $u=(-\p_x^2+V-z_0 I)^{-1}(f+V u)$
is in the range of $(-\p_x^2+V-z_0 I)^{-1}
\big(C\sb{\mathrm{comp}}(\R)\big)
\subset L^\infty(\R)$.

We point out that in the case of 
general exterior elliptic problems
%%with constant at infinity coefficients,
the above dichotomy
-- either boundedness of the truncated resolvent
%% $\unity\sb{\abs{x}<a}R(x,y;z)\unity\sb{\abs{y}<a}$
or existence of a nontrivial solution
to a homogeneous problem with
appropriate radiation conditions
--
was studied by B.\,Vainberg \cite{vainberg1975short}.
%%see e.g. \cite[Theorems~8,\,12,\,13]{vainberg1975short}

\begin{example}
%\noindent
%{\bf Example in $\R^3$.}
Here is an example of virtual levels
at $z_0=0$
of a Schr\"odinger operator in $\R^3$ from \cite{yafaev1975virtual}.
Let $u$ be a solution to
$-\Delta u+V u=0$ in $\R^3$.
Taking the Fourier transform,
we arrive at
$\hat u(\xi)=-\widehat{V u}(\xi)/\xi^2$.
The right-hand side
is not in $L^2_{\mathrm{loc}}(\R^3)$
if $\widehat{V u}(\xi)$ does not vanish at $\xi=0$;
this situation corresponds to
zero being a virtual level, with the corresponding
\emph{virtual state}
$u(x)\sim \abs{x}^{-1}$, $\abs{x}\gg 1$.
One can see that in the case of the
Schr\"odinger operator in $\R^3$
the space of virtual levels is at most one-dimensional.
A similar approach in two dimensions gives
\[
\hat u(\xi)
=-\frac{\widehat{V u}(\xi)}{\xi^2}
=-\frac{c_0+c_1\xi_1+c_2\xi_2+O(\xi^2)}{\xi^2},
\qquad
\xi\in\mathbb{B}^2_1,
\]
indicating that the space of virtual states
at $z_0=0$
of the Schr\"odinger operator in $\R^2$
could consist of up to one ``$s$-state''
approaching a constant value as $\abs{x}\to\infty$
and up to two ``$p$-states''
behaving like $\sim (c_1 x_1+c_2 x_2)/\abs{x}^2$
%%$\sim e^{\pm\jj\theta}/r$
for $\abs{x}\gg 1$.
\end{example}

\smallskip
\noindent
{\bf Relation to critical Schr\"odinger operators.\ }
In the context of positive-definite symmetric operators,
a dichotomy similar to 
having or not properties {\it (P1) -- (P3)}
-- namely,
either having a particular Hardy-type inequality
%%for the operator
or existence of a null state
%%of the corresponding quadratic form
--
is obtained
%by T.\,Weidl \cite[\S5.1]{weidl1999remarks},
by T.\,Weidl \cite{weidl1999remarks},
at that time a PhD. student of M.\,Birman and E.\,Laptev,
as a generalization of
Birman's approach \cite[\S1.7]{birman1961spectrum}
which was based on closures of the space with respect to
quadratic forms
corresponding to symmetric positive-definite operators
(in the spirit of the Krein--Vishik--Birman
extension theory
\cite{krein1947theory,vishik1952general,birman1956theory}).
This approach is directly related to the research
on subcritical and critical Schr\"odinger operators
\cite{simon1981large,murata1986structure,pinchover1988positive,pinchover1990criticality,gesztesy1991critical,pinchover2006ground,pinchover2007ground,takac2008generalized,devyver2014spectral,lucia2018criticality,lucia2020addendum}.
%(see also related work on
%the Hardy-type inequalities \cite{florkiewicz1999some}).
%%dealing with the existence of ``Hardy-type'' inequalities for Schr\"odinger operators.
Let us present the following result from \cite{pinchover2006ground},
which we write in the particular case
of $\Omega=\R^d$ and $V\in C_{\mathrm{comp}}(\R^d,\R)$:

\begin{borrowed}
Let
$H=-\Delta+V$
with $V\in C_{\mathrm{comp}}(\R^d,\R)$
%with $V\in L^p_{\mathrm{loc}}(\R^d,\R)$, $p>d/2$,
be a Schr\"odinger operator
in $L^2(\R^d)$,
and assume that
the associated quadratic form
\[
\mathbf{a}[u]
:=\int_{\R^d}(\abs{\nabla u}^2+V \abs{u}^2)\,dx
\]
is nonnegative on $C^\infty_{\mathrm{comp}}(\R^d)$.
Then either
%% $\bullet\ $
%% The Schr\"odinger operator
%% $H=-\Delta+V$
%% is called \emph{subcritical}
%% if it  has positive Green's function.
%% In this case,
%% there is a ``weighted gap condition'':
there is a continuous function $w(x)>0$
such that
$\int_{\R^d} w |u|^2\,dx\le
\mathbf{a}[u]
%:=\int_{\R^d}(\abs{\nabla u}^2+V \abs{u}^2)\,dx
$
for any $u\in C^\infty_{\mathrm{comp}}(\R^d)$
(one says that $\mathbf{a}[\cdot]$ has a \emph{weighted spectral gap}),
or
%\noindent
%$\bullet\ $
%If $H=-\Delta+V$ has no positive Green's function,
%but satisfies $A\ge 0$,
%then it is called \emph{critical}.
%In this case,
there is a sequence $\varphi_j\in C^\infty_{\mathrm{comp}}(\R^d)$
such that $\mathbf{a}[\varphi_j]\to 0$, $\varphi_j\to \varphi>0$
locally uniformly on $\R^d$
(then one says that $\mathbf{a}[\cdot]$ has a \emph{null state} $\varphi$).
\end{borrowed}

% \ac{NEW (August 2021)}

Let us mention that in the former case,
when $\mathbf{a}[\cdot]$ has a weighted spectral gap,
the operator $H$ is \emph{subcritical}
(that is, it admits a positive Green's function),
and that in the latter case,
when $\mathbf{a}[\cdot]$ has a null state,
$H$ is \emph{critical}.
This \emph{null state}
coincides with Agmon's \emph{ground state},
which can be characterized as a state with
minimal growth at infinity
from \cite[Definitions 4.1,\,5.1]{agmon1982positivity}.
See \cite{pinchover1988positive,pinchover1990criticality,pinchover2006ground}
for more details.

%\begin{remark}
%By \cite[Corollary 1.6]{pinchover2006ground},
%We refer to \cite{pinchover1990criticality}
%for more details.
%% {\it
%% Let $u$ be a positive solution of the equation
%% \begin{eqnarray}\label{agmon4.1}
%% Pu=0
%% \end{eqnarray}
%% in some neighborhood of infinity.
%% %% in $\Omega$.
%% We shall call $u$
%% a \emph{positive solution of minimal growth at infinity}
%% %if it has the following property.
%% if for any other positive solution $v$
%% to \eqref{agmon4.1}
%% in a neighborhood of infinity
%% there exists $C>0$
%% %and a neighborhood of infinity where both $u$ and $v$ are defined
%% such that $u\le Cv$ in some neighborhood of infinity
%% where both $u$ and $v$ are defined.
%% }
%\end{remark}

A \emph{null state}, or Agmon's \emph{ground state},
corresponds
to a virtual level at the bottom of the spectrum,
in the following sense:

\begin{lemma}
A nonnegative Schr\"odinger operator
$H=-\Delta+V$ in $L^2(\R^d)$,
with $V\in C_{\mathrm{comp}}(\R^d,\R)$,
has a null state $\varphi$ if
any compactly supported
% \ac{\sout{\nb{corrected } nonnegative}}
% \ac{$H-W$ is a negative perturbation of $H$ since $W$ is nonnegative! ??}
negative perturbation $H-W$ of $H$,
with $W\in C_{\mathrm{comp}}(\R^d)$,
$W\ge 0$, $W\ne 0$,
produces a negative eigenvalue.

% \ac{\sout{
% The \emph{only if} statement is also true
% under an additional assumption $V\le 0$
% and that $V$ and the perturbation are in $C^m_{\mathrm{comp}}(\R^d,\R)$,
% $m\ge\max(0,[n/2]-1)$ and supported in a fixed compact set
% $K\Subset\R^d$. 
% }}

% \ac{New version:}
For the converse,
we impose a stronger requirement that
$V\in C^m_{\mathrm{comp}}(\R^d,\R)$,
$m\ge\max(0,[n/2]-1)$,
$\supp V\subset K\Subset\R^d$.
If an arbitrary negative perturbation $H-W$ of $H$,
with
$W\in C^m_{\mathrm{comp}}(\R^d,\R)$,
$\supp W\subset K$,
$W\ge 0$, $W\ne 0$, produces a negative eigenvalue,
then $H$ has a null state.
% \ac{here we do not care whether $W$ is smooth.}
\end{lemma}

\begin{proof}
Let $\varphi>0$ be a null state of $H$ and let
$\varphi_j$ be a sequence such that
$\varphi_j\to \varphi$ locally uniformly on $\R^d$
and such that $\mathbf{a}[\varphi_j]\to 0$
as $j\to\infty$.
Let $W\in C_{\mathrm{comp}}(\R^d)$, $W\ge 0$, $W\not\equiv 0$.
Then
\[
\lim_{j\to\infty}
\langle\varphi_j,(H-W)\varphi_j\rangle
=
\lim_{j\to\infty}
\big(
\mathbf{a}[\varphi_j]
-
\langle\varphi_j,W\varphi_j\rangle
\big)
=
-\langle\varphi,W\varphi\rangle<0
\]
(we took into account the convergence $\varphi_j\to\varphi$,
locally uniformly on $\R^d$),
hence $\langle\varphi_j,(H-W)\varphi_j\rangle<0$
for some $j\in\N$,
and so the Rayleigh quotient for $H-W$ is
strictly negative,
leading to $\sigma(H-W)\cap\R_{-}\ne\emptyset$.

Let us prove the converse statement.
Let $\supp V\subset K\Subset\R^d$
and
let there be
perturbations $W_j\in C^m_{\mathrm{comp}}(\R^d,\R)$,
$j\in\N$,
with $\supp W_j\subset K$,
$W_j\ge 0$, $W_j\ne 0$
for all $j$,
and with $\sup_{x\in\R^d}\abs{\p_x^\beta W_j(x)}\to 0$
as $j\to\infty$
for all multiindices $\beta\in\N_0^d$ with
$\abs{\beta}\le m$.
%% (in the sense of the test functions),
By the assumption of the Lemma,
$\lambda_j:=\inf\sigma(H+W_j)<0$
(thus $\lambda_j\to 0-$ as $j\to\infty$).
Let $\psi_j\in L^2(\R^d,\C)$ be the corresponding eigenfunctions,
which can be shown to be from
$H^{m+2}(\R^d)\subset C^{\alpha}(\R^d)$,
$\forall\alpha<1/2$
(having the uniform bound in $H^{m+2}(\mathbb{B}^d_R,\C)$
for each $R\ge 1$).
By \cite[Theorem 8.38]{MR737190},
we can assume that $\psi_j$
are strictly positive.
Without loss of generality, we assume that
$\sup\sb{x\in\R^d}\psi(x)=1$.
By the maximum principle,
the functions
$\psi_j$ reach these maxima at some point
$x_j\in K$.
We may pass to a subsequence
so that
$x_j\to x_0\in K$ as $j\to\infty$.
Then, by the Ascoli--Arzel\`{a} theorem,
we may pass to a subsequence
so that the functions $\psi_j$ converge,
uniformly on compacts.
The limit function $\varphi\in C(\R^d,\C)$
is nonnegative
%%strictly positive and
and nonzero
(since
$\varphi(x_0)=\lim\sb{j\to\infty}\varphi_j(x_j)=1$),
and satisfies $H\varphi=0$ (in the sense of distributions).
%%AC????????
%%\ac{?????}
%Moreover, one can see that
%$\varphi(x)>0$ for all $x\in\R^d$.
%One can argue that $\mathbf{a}[\psi_j]\to 0$.
Since
\[
\mathbf{a}[\psi_j]+\langle\psi_j,W_j\psi_j\rangle
=\langle\psi_j,(H+W_j)\psi_j\rangle=\lambda_j\langle\psi_j,\psi_j\rangle<0
\qquad
\forall j\in\N,
\]
where $\langle\psi_j,W_j\psi_j\rangle\to 0$
(due to the convergence $\psi_j\to\varphi$,
$\supp W_j\subset K$,
and due to $\norm{W_j}_{L^\infty}\to 0$ as $j\to\infty$)
while $\mathbf{a}[\psi_j]\ge 0$,
one can see that
$\mathbf{a}[\psi_j]\to 0$.
%%and thus the limit function $\varphi$ is a \emph{null state}.
% \ac{\sout{
% Moreover, due to the assumption that
% $V\le 0$, any nonnegative solution to
% $H\varphi=0$ satisfies $\Delta(-\varphi)=-V\varphi\ge 0$,
% hence, by the strong maximum principle,
% $-\varphi\le 0$
% can not reach the maximum at any point of $\R^d$.}
% Added two sentences:}
Moreover, due to Harnack's inequality for
Schr\"odinger operators
\cite[Theorem 2.5]{chiarenza1986harnack},
since $\varphi(x_0)=1$,
one has $\varphi(x)\ne 0$ for all $x\in\R^d$.
(In \cite{chiarenza1986harnack}, the proof is given for $d\ge 3$
but is shown to apply to $d=2$ as well; the statement for $d=1$
is trivial by the ODE uniqueness theory.)
Thus the limit function $\varphi$ is a \emph{null state}.
% 
% \ac{The rest is not needed??}
% \nb{If $V=V^++V^-$ then 
% \[
% \mathbf{a}[u_j]\geq
% \mathbf{a^-}[u_j]:=\int_{\R^d}(\abs{\nabla u_j}^2+V^- \abs{u_j}^2)\,dx. 
% \]
% If $a^-$ is nonnegative so we conclude the same way.
% }
\end{proof}

%% By Theorem~\ref{theorem-b}~\itref{theorem-b-0} below,
%% this implies that there is no limiting absorption principle
%% for $H$ near $z_0=0$
%% relative to $\C\setminus\overline{\R_{+}}$, $\bfE=L^2_s(\R^d)$, and
%% $\bfF=L^2_{-s'}(\R^d)$, with arbitrarily large $s,\,s'>0$.
%See e.g. \cite{pinchover2006ground}.
%%See \cite{pinchover1990criticality}.

\begin{comment}
For example, $-\p_x^2$ in $L^2(\R)$
is \emph{critical};
see Example~\ref{example-phi}.
\end{comment}

\section{General theory of virtual levels
in Banach spaces}
\label{sect-general}

% \ac{Reformulated before Definition~\ref{def-virtual}}
We now sketch our approach to virtual levels
from \cite{virtual-levels}.
Let $\bfX$ be an infinite-dimensional complex Banach space and
let $A\in\scrC(\bfX)$
be a closed operator with dense domain $\dom(A)\subset\bfX$.
We assume that there are some complex Banach spaces $\bfE$, $\bfF$
with embeddings
$\bfE\hookrightarrow\bfX\hookrightarrow\bfF$.
We will assume that the operator $A$
and the ``regularizing'' spaces $\bfE$ and $\bfF$
satisfy the following assumption.

\begin{assumption}
\label{ass-virtual}
\begin{enumerate}
\item
\label{ass-virtual-1}
The embeddings
\[
\bfE\mathop{\longhookrightarrow}\limits\sp{\imath}
\bfX\mathop{\longhookrightarrow}\limits\sp{\jmath}\bfF
\]
are dense and continuous.
\item
\label{ass-virtual-2}
The operator $A:\,\bfX\to\bfX$, considered as a mapping
$\bfF\to\bfF$,
\[
A\sb{\bfF\shortto\bfF}:\,\bfF\to\bfF,
\quad
\dom(A\sb{\bfF\shortto\bfF})=\jmath(\dom(A)),
\quad
A\sb{\bfF\shortto\bfF}:\,y\mapsto \jmath(A x)
\;\; \mbox{if}\;\; y=\jmath(x),
\]
is closable in $\bfF$, with closure $\hatA\in\scrC(\bfF)$
and domain
$\dom(\hatA)\supset\jmath\big(\dom(A)\big)$.
\item
\label{ass-virtual-3}
Denote
\[
 \dom(A_{\bfE\shortto\bfE})
%%???=\{\phi\in \dom(A)\cap \imath(\bfE),\,A\phi\in \imath(\bfE)\}
=\{\phi\in\bfE\sothat\,\imath(\phi)\in\dom(A),\,A\imath(\phi)\in\imath(\bfE)\}
\]
and 
\[
 \dom(\hatA_{\bfE\shortto\bfE})
%%%???=\{\phi\in \dom(\hatA)\cap \jmath\circ\imath(\bfE),\,\hatA\phi\in \jmath\circ\imath(\bfE)\}.
=\{\phi\in\bfE\sothat\,\jmath\circ\imath(\phi)\in\dom(\hatA),\,\hatA\jmath\circ\imath(\phi)\in \jmath\circ\imath(\bfE)\}.
\]
The space $\jmath\circ\imath(\dom(A_{\bfE\shortto\bfE}))$ is dense in
$\jmath\circ\imath(\dom(\hatA_{\bfE\shortto\bfE}))$
in the topology induced by the graph norm of $\hatA$,
defined by
\[
\norm{\psi}_{\hatA}
=
\norm{\psi}_{\bfF}
+
\norm{\hatA\psi}_{\bfF},
\qquad
\psi\in\dom(\hatA)\subset\bfF.
\]

% % The set $\jmath\big(\{\imath(\phi)\in\dom(A),\,(A-z_0 I) \imath(\phi)\in \imath(\bfE)\}\big)$ is dense in 
% % $\{\phi\in\dom(\hatA),\,(\hatA-z_0 I)\phi\in \jmath\circ(\bfE)\}$.
\end{enumerate}
\end{assumption}

We note that Assumption~\ref{ass-virtual}
is readily satisfied in the usual examples
of differential operators.
For convenience,
from now on,
we will assume that
$\bfE\subset\bfX\subset\bfF$ (as vector spaces)
%so that we can omit
and will omit
$\imath$ and $\jmath$
in numerous relations.

%% Let
%% \[
%% \varOmega\subset\C\setminus\sigma(A)
%% \]
%% be a connected open set such that
%% $\sigma\sb{\mathrm{ess}}(A)\cap\p\varOmega$ is nonempty.

\begin{definition}[Virtual levels]
\label{def-virtual}
Let $A\in\scrC(\bfX)$ and $\bfE\hookrightarrow\bfX\hookrightarrow\bfF$
satisfy Assumption~\ref{ass-virtual}.
%% and $\bfE\hookrightarrow\bfX\hookrightarrow\bfF$ (densely and continuously).
Let
\[
\varOmega\subset\C\setminus\sigma(A)
\]
be a connected open set
such that $\sigma\sb{\mathrm{ess}}(A)\cap\p\varOmega$ is nonempty.
We say that a point
$z_0\in\sigma\sb{\mathrm{ess}}(A)\cap\p\varOmega$
is a
\emph{point of the essential spectrum of $A$
of rank $r\in\N_0$
relative to $(\varOmega,\bfE,\bfF)$
}
if it is the smallest value
for which there is
$B\in\scrB_{00}(\bfF,\bfE)$
(with $\scrB_{00}$ denoting bounded operators of finite rank)
of rank $r$
such that
\begin{eqnarray}
\varOmega\cap\sigma(A+B)\cap\mathbb{D}_\delta(z_0)=\emptyset
\end{eqnarray}
for some $\delta>0$,
and there exists the following limit
in the weak operator topology
of mappings $\bfE\to\bfF$:
\begin{eqnarray}\label{lim}
(A+B-z_0 I)_{\varOmega,\bfE,\bfF}^{-1}:=
\wlim\sb{z\to z_0,\,z\in\varOmega}(A+B-z I)^{-1}
:\;\bfE\to\bfF.
\end{eqnarray}
Points of rank $r=0$
relative to $(\varOmega,\bfE,\bfF)$
(so that there is a limit \eqref{lim} with $B=0$)
are called
\emph{regular points of the essential spectrum
relative to $(\varOmega,\bfE,\bfF)$}.
%% We denote the set of such points by $\sigma\sb{\varOmega}^0(A)$.

If $z_0$
is of rank $r\ge 1$
relative to $(\varOmega,\bfE,\bfF)$,
we call it
an \emph{exceptional point of rank $r$ relative to $(\varOmega,\bfE,\bfF)$},
or a \emph{virtual level of rank $r$ relative to $(\varOmega,\bfE,\bfF)$}.
The corresponding \emph{virtual states}
are defined as elements of the space
\[
\frakM_{\varOmega,\bfE,\bfF}(A-z_0 I)
:=\big\{
\Psi\in\range\big((A+B-z_0 I)^{-1}_{\varOmega,\bfE,\bfF}\big)
\sothat
(A\sb{\bfF\shortto\bfF}-z_0 I)\Psi=0
\big\},
\]
with any $B\in\scrB_{00}(\bfF,\bfE)$
such that
the limit \eqref{lim}
is defined
(this space is of dimension $r$ and does not depend on the choice of
$B$; see Theorem~\ref{theorem-m} below).
\end{definition}

Above,
$\sigma\sb{\mathrm{ess}}(A)$
is F.\, Browder's essential spectrum
\cite[Definition 11]{browder1961spectral}.
It can be characterized as
$\sigma(A)\setminus\sigma\sb{\mathrm{d}}(A)$,
with the discrete spectrum $\sigma\sb{\mathrm{d}}(A)$
being the set of isolated points of $\sigma(A)$
with corresponding Riesz projectors having finite rank
(see e.g. \cite[Lemma~III.125]{opus}).
Let us emphasize that the existence of the limit \eqref{lim}
implicitly implies that there is $\delta>0$ such that
$\varOmega\cap\sigma(A+B)\cap\mathbb{D}_\delta(z_0)=\emptyset$.

\begin{comment}
Thus, when we consider the equation
\begin{eqnarray}\label{a-u-f}
(A-z_0)u=f,
\end{eqnarray}
either there
exists a limit \eqref{lim} with $B=0$
(``$z_0$ is of rank $r=0$''),
so that there is a (unique) solution to \eqref{a-u-f}
given by
\begin{eqnarray}\label{certain}
u=(A-z_0 I)^{-1}_{\varOmega,\bfE,\bfF}f\in\bfF,
\end{eqnarray}
with
$
(A-z_0 I)^{-1}_{\varOmega,\bfE,\bfF}
:=
\wlim\sb{z\to z_0,\,z\in\varOmega}
(A-z I)^{-1}:\,\bfE\to\bfF$,
%in the sense of mappings
%(it turns out that ``in the weak operator topology'' is enough here),
or there is a nontrivial solution to
\[
(A-z_0)u=0,
\qquad u\in\bfF,
\]
with the additional assumption that
$u\in\range\big((A+B-z_0 I)^{-1}_{\varOmega,\bfE,\bfF}\big)$
(with $\range$ denoting the range),
with some $B$ such that the limit
\eqref{lim}
makes sense.
%%(again, in the weak operator topology).
Indeed, fixing such $B$,
given a solution to $(A-z_0)u=0$,
one has:
\[
(A+B-z_0)u=B u,
\ \quad
u=(A+B-z_0 I)^{-1}_{\varOmega,\bfE,\bfF} B u.
%%\in\range\big((A+B-z_0 I)^{-1}_{\varOmega,\bfE,\bfF}\big).
\]
\end{comment}

\begin{remark}
Definition~\ref{def-virtual}
allows one to treat
generalized eigenfunctions
corresponding to ``threshold resonances''
of a Schr\"odinger operator
$A$
(not necessarily selfadjoint)
and solutions to $(A-z_0 I)u=0$
with $z_0$ from the bulk of
$\sigma\sb{\mathrm{ess}}(A)$
which satisfy the \emph{Sommerfeld radiation condition}
as the same concept
of \emph{virtual states}
$\Psi\in\frakM_{\varOmega,\bfE,\bfF}(A-z_0 I)$
(with appropriate choice of $\varOmega$).
\end{remark}

\begin{remark}
In case when $z_0$ is a virtual level
but not an eigenvalue,
it seems reasonable to call it an \emph{(embedded) resonance}.
%or perhaps \emph{embedded resonance}.
Note that the name \emph{threshold resonance}
seems misleading,
since in the nonselfadjoint case
a virtual level could be located
%%anywhere in the essential spectrum (or, rather,
at any point of contact of the essential spectrum
with the resolvent set,
not necessarily at a \emph{threshold}.
(According to \cite{howland1974poiseux},
thresholds could be defined as
(i) a branch point of an appropriate function,
(ii) a point where the absolutely continuous part changes multiplicity,
or (sometimes) (iii) an end point of the spectrum.)
% \nb{Actually even in the selfadjoint case, thresholds can be embedded and hence difficult to define. Maybe this is what the parenthesis means ...}
% \ac{This is what I think...
% Possibly in  \cite{howland1974poiseux} this is all about selfadjoit case
% since he talks about scattering...
% }
\end{remark}

\begin{remark}
\label{remark-null}
The dimension of the null space of a square matrix
$M%\in\End(\C^N)
$
can be similarly characterized
as
$\dim\ker(M)=\min
\big\{
\rank N\sothat\det(M+N)\ne 0
\big\}.
$
For example,
for
$M=
\footnotesize
\tiny
\begin{bmatrix}0&1&0\\0&0&1\\0&0&0\end{bmatrix}
$,
we can take
$N=
\footnotesize
\tiny
%\scriptsize
\begin{bmatrix}0&0&0\\0&0&0\\1&0&0\end{bmatrix}
$,
in agreement with $\dim\ker(M)=1$.
\end{remark}

\begin{example}
Let $A=-\Delta$ in $L^2(\R^3)$, $\dom(A)=H^2(\R^3)$.
By \cite[Appendix A]{agmon1975spectral},
for any $s,\,s'>1/2$
and $z_0>0$,
the resolvent $(-\Delta-z I)^{-1}$ converges
as $z\to z_0\pm 0\jj$
in the uniform operator topology of continuous mappings
$L^2_s(\R^3)\to L^2_{-s'}(\R^3)$.
The two limits differ;
%depending whether $z\to z_0\pm 0\jj$;
the integral kernels of the limiting operators
$(-\Delta-z_0 I)_{\C_\pm}^{-1}$
are given by
$e^{\pm\jj\abs{x-y}\sqrt{z_0}}/(4\pi\abs{x-y})$,
$x,\,y\in\R^3$.
It follows that $z_0>0$ is a regular point
of the essential spectrum of $-\Delta$
relative to $\varOmega=\C_{\pm}$.
Moreover,
according to \cite{MR544248},
there is a limit of the resolvent
as $z\to z_0=0$, $z\in\C\setminus\overline{\R_{+}}$,
in the uniform operator topology of continuous mappings
$L^2_s(\R^3)\to L^2_{-s'}(\R^3)$,
$s,\,s'>1/2$, $s+s'>2$,
hence $z_0=0$ is also a regular point of the essential spectrum
(relative to $\varOmega=\C\setminus\overline{\R_{+}}$).
\end{example}

%% \begin{example}
%% $z_0=0$ is a regular point of
%% $A=-\Delta$ in $L^2(\R^3)$ with $\dom(A)=H^2(\R^3)$
%% relative to $\varOmega=\C\setminus\overline{\R_{+}}$,
%% since its resolvent
%% $(-\Delta-z I)^{-1}$, $z\in\C\setminus\overline{\R_{+}}$,
%% with the integral kernel
%% \eqref{free-resolvent-3d}
%% %% \[
%% %% R_0^{(3)}(x,y;z)=\frac{e^{-\sqrt{-z}\abs{x-y}}}{4\pi\abs{x-y}},
%% %% \qquad
%% %% z\in\C\setminus\overline{\R_{+}},
%% %% \qquad
%% %% \Re\sqrt{-z}>0,
%% %% %\qquad
%% %% %\abs{R_0^{(3)}(x,y;z)}
%% %% %\le\frac{1}{4\pi\abs{x-y}},
%% %% \]
%% extends to a bounded linear mapping
%% $L^2_s(\R)\to L^2_{-s'}(\R)$,
%% with any fixed $s,\,s'>1/2$, $s+s'\ge 2$,
%% and for $s+s'>2$ there is a limit (in the uniform operator topology)
%% \ac{Not necessarily so if $s+s'=2$...}
%% of $(-\Delta-z I)^{-1}$
%% as $z\to z_0=0$, $z\in\C\setminus\overline{\R_{+}}$
%% (see Lemmata~\ref{lemma-f2} and~\ref{lemma-f3} below).
%% \end{example}

\begin{example}
Consider the differential operator
$A=-\jj\p_x +V:\,L^2(\R)\to L^2(\R)$,
$\dom(A)=H^1(\R)$,
with $V$ the operator of multiplication by $V\in L^1(\R)$.
The solution to
$(-\jj\p_x+V-z I)u=f\in L^1(\R)$,
$z\in\C_{+}$,
is given by
\[
u(x)=\jj\!\int\limits^x_{-\infty}\!e^{\jj z(x-y)-\jj W(x)+\jj W(y)}f(y)\,dy,
\quad
W(x):=\!\int\limits^x_{-\infty}\! V(y)\,dy,\; W\in L^\infty(\R).
\]
For each $z\in\C_{+}$,
the mapping
$(A-z I)^{-1}:\,f\mapsto u$
is continuous
%% from $L^2(\R)$ to $L^2(\R)$
%%(hence $\sigma(A)\cap\C_{+}=\emptyset$; similarly, $\sigma(A)\cap\C_{-}=\emptyset$)
%% and
from $L^1(\R)$ to $L^\infty(\R)$,
with the
%% $L^1\to L^\infty$
bound uniform in $z\in\C_{+}$.
%and hence
%$L^2_s(\R)\to L^2_{-s'}(\R)$, $s,\,s'>1/2$,
% and
Moreover, one can see that for each $z_0\in\R$ there exists a limit
$(A-z_0 I)^{-1}_{\C_{+},L^1,L^\infty}=
\lim\limits_{z\to z_0,\ z\in\C_+}
(A-z I)^{-1}$
in the strong operator topology of mappings
$L^1\to L^\infty$;
thus, any $z_0\in\R$
is a regular point of the essential spectrum
relative to $(\C_{+},\,L^2_s(\R),\,L^2_{-s'}(\R))$
(and similarly relative to $(\C_{-},\,L^2_s(\R),\,L^2_{-s'}(\R))$).
\end{example}

\begin{example}\label{example-shift}
Consider the left shift
$L:\,\ell^2(\N)\to \ell^2(\N)$,
$(x_1,x_2,x_3,\dots)
\mapsto
(x_2,x_3,x_4,\dots)$,
with
$
\sigma(L)=\sigma\sb{\mathrm{ess}}(L)
=\overline{\mathbb{D}_1}$.
The matrix representations of $L-z I$
and
$(L-z I)^{-1}$,
$\abs{z}>1$,
are given by
%% \[
%% L-z I=
%% {\footnotesize
%% \begin{bmatrix}
%% -z&1&0&\cdots\\
%% 0&-z&1&\cdots\\
%% 0&0&-z&\cdots\\
%% \vdots&\vdots&\vdots&\ddots
%% \end{bmatrix}
%% },
%% \ \ z\in\C;
%% \quad
%% (L-z I)^{-1}
%% =
%% -
%% {\footnotesize
%% \begin{bmatrix}
%% z^{-1}&z^{-2}&z^{-3}&\cdots\\
%% 0&z^{-1}&z^{-2}&\cdots\\
%% 0&0&z^{-1}&\cdots\\
%% \vdots&\vdots&\vdots&\ddots
%% \end{bmatrix}
%% },
%% \ \ z\in\C\setminus\overline{\mathbb{D}_1}.
%% \]
\[
L-z I=
{\footnotesize
\tiny
\begin{bmatrix}
-z&1&0&\cdots\\[1pt]
0&-z&1&\cdots\\[1pt]
0&0&-z&\cdots\\[1pt]
\cdots&\cdots&\cdots&\cdots
\end{bmatrix}
},
\ z\in\C;
\quad
(L-z I)^{-1}
=
-
{\footnotesize
\tiny
\begin{bmatrix}
z^{-1}&z^{-2}&z^{-3}&\cdots\\
0&z^{-1}&z^{-2}&\cdots\\
0&0&z^{-1}&\cdots\\
\cdots&\cdots&\cdots&\cdots
\end{bmatrix}
},
\ z\in\C\setminus\overline{\mathbb{D}_1}.
\]
   From the above representation,
one has
$\abs{((L-z I)^{-1}x)_i}\le
\abs{z^{-1}x_i}+\abs{z^{-2}x_{i+1}}+\dots\le\norm{x}_{\ell^1}$,
and moreover $\lim_{i\to\infty}((L-z I)^{-1}x)_i=0$, 
for any $x\in \ell^1(\N)\subset \ell^2(\N)$ and any $z\in\C$, $\abs{z}>1$,
hence
$(L-z I)^{-1}$
defines a continuous linear mapping
$\ell^1(\N)\to c_0(\N)$,
with the norm bounded (by one) uniformly
in $z\in\C$, $\abs{z}>1$.
For any $\abs{z_0}=1$,
the mappings
$(L-z I)^{-1}:\,\ell^1(\N)\to c_0(\N)$
have a limit
as $z\to z_0$, $\abs{z}>1$,
in the weak operator topology
(also in the strong operator topology).
It follows that any
of the boundary points of the spectrum
of $L$
(i.e., any $z_0\in\C$ with $\abs{z_0}=1$)
is a regular point of the essential spectrum
relative to
$\big(\C\setminus\overline{\mathbb{D}_1},\,\ell^1(\N),\,c_0(\N)\big)$.

Let us construct an operator with a virtual level
at $z_0\in\C$, $\abs{z_0}=1$.
Assume that
$K\in\scrB_{00}\big(\ell^\infty(\N),\ell^1(\N)\big)$
has eigenvalue $1\in\sigma\big(K\at{\ell^1}\big)$,
with the corresponding eigenfunction $\phi\in \ell^1(\N)$.
Then the operator
$A=L-K(L-z_0 I)$,
$\dom(A)=\ell^2(\N)$,
has a virtual level at $z_0$
since $z_0$ is a regular point of $A+B$,
with $B=K(L-z_0 I):\,c_0(\N)\to \ell^1(\N)$
of finite rank
(we note that $L$ has a bounded extension onto $c_0(\N)$).
The function
$\Psi=(L-z_0 I)^{-1}_{\varOmega,\ell^1,c_0}\phi\in c_0(\N)$
is a virtual state of
$A=L-K(L-z_0 I)$ corresponding to $z_0$,
relative to
$\big(\C\setminus\overline{\mathbb{D}_1},\,\ell^1(\N),\,c_0(\N)\big)$,
satisfying $(\hatA-z_0 I)\Psi=0$,
%(cf. Lemma~\ref{lemma-jk}~\itref{lemma-jk-1} below),
with $\hatA$ a closed extension of $A$
onto $c_0(\N)$.
\end{example}

\begin{example}\label{example-c}
Let $\bfX$ be an infinite-dimensional Banach space
and let $Z:\,\bfX\to\bfX$, $\psi\mapsto 0$, $\forall \psi\in\bfX$,
be the zero operator with
$\sigma(Z)=\sigma\sb{\mathrm{ess}}(Z)=\{0\}$.
Assume that $\bfE,\,\bfF$ are Banach spaces
with dense continuous embeddings
$\bfE\hookrightarrow\bfX\hookrightarrow\bfF$.
Let $B\in\mathscr{B}_{00}(\bfF,\bfE)$.
Let $\epsilon>0$ be smaller than the absolute value
of the smallest nonzero eigenvalue of $B$
(there are finitely many nonzero eigenvalues
since $B$ is of finite rank),
and define
\[
P_0=
-\frac{1}{2\pi\jj}
\ointctrclockwise_{\abs{\zeta}=\epsilon}
(B-\zeta I)^{-1}\,d\zeta
:\;\bfX\to\bfX
\]
to be a projection onto the eigenspace of $B$
corresponding to eigenvalue $\lambda=0$.
%%$\range(B)\subset\bfX$.
Then,
for $z\in\C\setminus\{0\}$,
\[
(Z+B-z I)^{-1}%%(I-P)
P_0
=(B-z I)^{-1}%%(I-P)
P_0
=-z^{-1}%%(I-P)
P_0:\,\bfX\to\bfX,
\]
hence
\[
(Z+B-z I)^{-1}%%(I-P)
P_0
=-z^{-1}
%%(I-P)
P_0
:\,\bfE\to\bfF,
\quad z\in\C\setminus\{0\},
\]
with the norm unbounded as $z\to 0$, $z\ne 0$.
Thus, $z_0=0$ is an exceptional point
of the essential spectrum of $Z$
of infinite rank
relative to $\C\setminus\{0\}$
and arbitrary $\bfE,\,\bfF$.
\end{example}

%\medskip

%\noindent
%{\bf Spectral singularities.}
\begin{remark}
Let us contrast virtual levels
to \emph{spectral singularities}
\cite{naimark1954investigation,schwartz1960some,pavlov1966nonselfadjoint,lyantse1967differential,guseinov2009concept,konotop2019designing}
(for a more general setting, see \cite{nagy1986operators}).
%% which can be defined
%% in the context of (nonselfadjoint) Schr\"odinger operators
%% as the points of the essential spectrum
%% near which the norm of the resolvent grows faster than
%% $\sim 1/\dist(z,\sigma_{\mathrm{ess}})$.
%or the points where the integral kernel of the resolvent
%has no limit.
%See also \cite{pavlov1966nonselfadjoint,lyantse1967differential,
%guseinov2009concept,konotop2019designing};
%for a more general setting, see \cite{nagy1986operators}.
%Whie selfadjoint Schr\"odinger operators
%$H=-\Delta+V$, $V\in C_{\mathrm{comp}}(\R^d,\R)$,
%% For example,
%% in the case of Schr\"odinger operators in one dimension,
%% the spectral singularities and virtual levels at
%% the bulk of the essential spectrum
%% admit the same characterization as zeros of the
%% Wronskian of the corresponding Jost solutions.
%At the same time,
We note that
selfadjoint operators have no spectral singularities,
although they could have virtual levels at threshold points;
this shows that these two concepts differ.
%% a virtual level at $z_0=0$
%% is not a spectral singularity
%since for a selfadjoint operator
%one has $\norm{(H-z I)^{-1}}=1/\dist(z,\sigma(H))$.
%% \ac{
%% INDEED, LET US REMOVE THE FOLLOWING PART
%% WHICH IS TOO CONFUSING AND BARELY CORRECT:
%% \sout{On the other hand,
%% given $H_0=-\Delta$ in $\R^d$, $d\ge 1$,
%% and a complex-valued potential
%% $V\in C\sb{\mathrm{comp}}(\R^d,\C)$,
%% from the decomposition}
%% \begin{eqnarray}\label{h-v-h}
%% (H_0+V-z I)^{-1}=(I+(H_0-z I)^{-1}V)^{-1}(H_0-z I)^{-1},
%% \ \quad
%% z\in\C_{+},
%% \end{eqnarray}
%% \sout{
%% one can see that
%% if $-1\in\sigma\big((H_0-z_0 I)^{-1}_{\C_{+}}V\big)$,
%% where $(H_0-z_0 I)^{-1}_{\C_{+}}$ is defined as the limit
%% of $(H_0-z I)^{-1}$,
%% as $z\to z_0$ with $z\in\C_{+}$
%% (we note that the above \nb{\sout{condition} changed } limit cannot be satisfied
%% if $V$ is real-valued and $z_0\in\R^+$ \nb{Actually not clear ...}),
%% then the norm of the operator in the right-hand side
%% of \eqref{h-v-h}
%% can grow faster than the distance from $z\in\C_{+}$
%% to $\sigma(H)$
%% and thus $z_0$ is both a virtual level of $H_0+V$
%% (relative to $\C_{+}$)
%% and
%% a spectral singularity.}
%% }
%% \nb{Then we should remove this remark ?}
\end{remark}

%\medskip

%\noindent
%{\bf Pseudospectrum.}
\begin{remark}
There is no
direct relation of virtual levels to \emph{pseudospectrum}
\cite{landau1975szego}.
For $A\in\scrC(\bfX)$, one defines
the $\varepsilon$-pseudospectrum
by
\[
\sigma_\varepsilon(A)
=\sigma(A)\cup
\{
z\in\C\setminus\sigma(A)\sothat
\norm{(A-z I)^{-1}}\ge\varepsilon^{-1}
\}.
\]
Since $\sigma\sb\varepsilon\big(-\Delta\at{L^2(\R^d)}\big)=
\{
z\in\C;\,
\dist\big(z,\overline{\R_{+}}\big)\le\varepsilon\}$
does not depend on the dimension $d\ge 1$,
the behaviour of pseudospectrum
%%of the Laplacian in $\R^d$
near the threshold $z_0=0$
does not distinguish the presence of
a virtual level at $z_0$ for $d\le 2$
and its absence for $d\ge 3$.
\end{remark}

The following key lemma
is essentially an abstract version of \cite[Lemma 2.4]{MR544248}.

\begin{lemma}[Limit of the resolvent
as the left and right inverse]
\label{lemma-jk}
%% Let $A\in\scrC(\bfX)$,
%% $\varOmega\subset\C\setminus\sigma(A)$,
%% and $\bfE\hookrightarrow\bfX\hookrightarrow\bfF$ (densely and continuously).
Let $A\in\scrC(\bfX)$ and $\bfE\hookrightarrow\bfX\hookrightarrow\bfF$
satisfy Assumption~\ref{ass-virtual}.
Let $\varOmega\subset\C\setminus\sigma(A)$.
Assume that $z_0\in\sigma\sb{\mathrm{ess}}(A)\cap\p\varOmega$
is a regular point of the essential spectrum
relative to $(\varOmega,\bfE,\bfF)$,
so that there exists a limit
\[
(A-z_0 I)^{-1}_{\varOmega,\bfE,\bfF}
:=
\wlim\sb{z\to z_0,\,z\in\varOmega}
(A-z I)^{-1}:\;\bfE\to\bfF.
\]
This limit
is both the left and the right inverse of
$\hatA-z_0 I:\range\big((A-z_0 I)^{-1}_{\varOmega,\bfE,\bfF}\big)\to \bfE$.
\end{lemma}

In applications
one needs to consider
not only finite rank perturbations but also
relatively compact perturbations,
allowing in place of $B$ in \eqref{lim}
operators which are $\hatA$-compact,
in the following sense.

\begin{definition}\label{def-compact}
Let $\hatA:\,\bfF\to\bfF$
and $B:\,\bfF\to\bfE$ be linear,
with
$\dom(B)\supset\dom(\hatA)$.
We say that
$B$ is $\hatA$-compact if
$\range\big(B\big\vert\sb{
\{x\in\dom(\hatA);
\,\norm{x}_\bfF^2+\norm{\hatA x}_\bfF^2\le 1
\}}
\big)\subset\bfE$
is precompact.
\end{definition}

We denote the set of $\hatA$-compact operators
%$B\in\scrC(\bfF,\bfE)$
for which the limit \eqref{lim} exists
by
\begin{multline*}
\scrQ_{\varOmega,\bfE,\bfF}(A-z_0 I)=
\big\{
%%B\in\scrC(\bfF,\bfE);
\mbox{$B$ is $\hatA$-compact;\ }
\,
\exists \delta>0, \varOmega\cap\sigma(A+B)\cap\mathbb{D}_\delta(z_0)=\emptyset,\\
\exists
\!
\wlim\sb{z\to z_0,\,z\in\varOmega}
(A+ B-z I)^{-1}:\,\bfE\to\bfF
\big\}.
\end{multline*}

\begin{comment}
This set is open, in the following sense:

\begin{lemma}\label{lemma-relatively-open}
Let $B\in\scrQ_{\varOmega,\bfE,\bfF}(A-z_0 I)$.
If $B_1\in\scrC(\bfF,\bfE)$
is $\hatA$-compact and satisfies
$
\norm{(B-B_1)(A+B-z_0 I)_{\varOmega,\bfE,\bfF}^{-1}}_{\bfE\to\bfE}
<1$,
then
$B_1\in\scrQ_{\varOmega,\bfE,\bfF}(A-z_0 I)$.
\end{lemma}
\end{comment}

\begin{theorem}[Independence from the regularizing operator]
\label{theorem-lap}
%% Let $A\in\scrC(\bfX)$,
%% $\varOmega\subset\C\setminus\sigma(A)$,
%% and $\bfE\hookrightarrow\bfX\hookrightarrow\bfF$ (densely and continuously).
Let $A\in\scrC(\bfX)$ and $\bfE\hookrightarrow\bfX\hookrightarrow\bfF$
satisfy Assumption~\ref{ass-virtual}.
Let $\varOmega\subset\C\setminus\sigma(A)$.
Assume that
$z_0\in\sigma\sb{\mathrm{ess}}(A)\cap\p\varOmega$
is a regular point of the essential spectrum
relative to $(\varOmega,\bfE,\bfF)$,
so that there is a limit
$
(A-z_0 I)_{\varOmega,\bfE,\bfF}^{-1}
:=
\wlim\sb{z\to z_0,\,z\in\varOmega}
(A-z I)^{-1}:\,\bfE\to\bfF.
$
Assume that
$B\in\scrC(\bfF,\bfE)$
is $\hatA$-compact.
Then:
\begin{enumerate}
\item
\label{theorem-lap-2}
For each $B\in\scrC(\bfF,\bfE)$ which is $\hatA$-compact
and such that there exists $\delta>0$
which satisfies
$\varOmega\cap\sigma(A+B)\cap\mathbb{D}_\delta(z_0)=\emptyset$,
the following statements are equivalent:
\begin{enumerate}
\item
\label{theorem-lap-2-a}
There is no nonzero solution to
$(\hatA+B-z_0 I)\Psi=0$,
$\Psi\in\range\big((A-z_0 I)^{-1}_{\varOmega,\bfE,\bfF}\big)$;
\begin{comment}
\item
\label{theorem-lap-2-b}
$-1\not\in\sigma\sb{\mathrm{p}}\big(B(A-z_0 I)^{-1}_{\varOmega,\bfE,\bfF}\big)$,
where $B(A-z_0 I)^{-1}_{\varOmega,\bfE,\bfF}$
is considered in $\bfE$;
\item
\label{theorem-lap-2-c}
$-1\not\in\sigma\sb{\mathrm{p}}\big((A-z_0 I)^{-1}_{\varOmega,\bfE,\bfF}B\big)$,
where $(A-z_0 I)^{-1}_{\varOmega,\bfE,\bfF}B$
is considered in $\bfF$;
\end{comment}
\item
\label{theorem-lap-2-d}
There 
%% \nb{\sout{
%% is $\delta>0$ such that
%% $\varOmega\cap\mathbb{D}_\delta(z_0)\subset\C\setminus\sigma(A+B)$,
%% and there}}
exists a limit
\[
(A+B-z_0 I)^{-1}_{\varOmega,\bfE,\bfF}
:=\wlim\sb{z\to z_0,\,z\in\varOmega\cap\mathbb{D}_\delta(z_0)}
(A+B-z I)^{-1}:\;\bfE\to\bfF.
\]
\textup(That is, there is the inclusion
$B\in\scrQ_{\varOmega,\bfE,\bfF}(A-z_0 I)$.\textup)
\end{enumerate}
\item
\label{theorem-lap-3}
If any (and hence both)
of the statements from
Part\,1
is satisfied,
then:
\begin{enumerate}
\item
\label{theorem-lap-3-a}
%\begin{eqnarray}\label{r-is-r}
$
\range\big((A-z_0 I)^{-1}_{\varOmega,\bfE,\bfF}\big)
=
\range\big((A+B-z_0 I)^{-1}_{\varOmega,\bfE,\bfF}\big);
$
%\end{eqnarray}
\item
\label{theorem-lap-3-b}
If the operators
$(A-z I)^{-1}$ converge
as $z\to z_0$, $z\in\varOmega$,
in the strong or uniform
operator topology of mappings $\bfE\to\bfF$,
then 
$(A+B-z I)^{-1}$ converge
as $z\to z_0$, $z\in\varOmega$,
in the same topology;
\item
\label{theorem-lap-3-c}
If
there are Banach spaces $\bfE_0$ and $\bfF_0$
with dense continuous embeddings
$\bfE\hookrightarrow\bfE_0\hookrightarrow\bfX\hookrightarrow
\bfF_0\hookrightarrow\bfF$,
such that the operator
$(A-z_0 I)^{-1}_{\varOmega,\bfE,\bfF}$
extends to a bounded mapping
$
(A-z_0 I)^{-1}_{\varOmega,\bfE,\bfF}:\;\bfE_0\to\bfF_0$,
then
$(A+B-z_0 I)^{-1}_{\varOmega,\bfE,\bfF}$
also extends to a bounded mapping
$\bfE_0\to\bfF_0$.
\end{enumerate}
\end{enumerate}
\end{theorem}

\begin{remark}
Regarding Theorem~\ref{theorem-lap}~\itref{theorem-lap-3-c},
it is possible that
$(A+B-z_0 I)^{-1}_{\varOmega,\bfE,\bfF}$
extends to a bounded map
$\bfE_0\to\bfF_0$,
yet
there is no convergence $(A+B-z I)^{-1}\to (A+B-z_0 I)^{-1}_{\varOmega,\bfE,\bfF}$
in the weak operator topology of mappings $\bfE_0\to\bfF_0$.
For example,
the resolvent of the Laplacian in $\R^d$, $d\ge 5$,
converges in the weak operator topology of continuous linear mappings
$L^2_s(\R^d)\to L^2_{-s'}(\R^d)$, $s+s'>2$,
as $z\to z_0=0$, $z\in\C\setminus\overline{\R_{+}}$,
only as long as $s,\,s'>1/2$,
while the limit
$(-\Delta-z_0 I)_{\varOmega}^{-1}$
extends to continuous mappings
%% continuous linear mappings
$L^2_2(\R^d)\to L^2(\R^d)$,
$L^2(\R^d)\to L^2_{-2}(\R^d)$.
\end{remark}

Now we
introduce the space of virtual states
$\frakM$.
This space appears
in \cite{MR544248}
in the context of Schr\"odinger operators in $\R^3$
(see also \cite[\S1.7]{birman1961spectrum}).
%%and in \cite{yajima2005dispersive}

%% and
%% shows that virtual levels
%% could be characterized as points $z_0$
%% such that the limit in \eqref{lim}
%% exists not only for finite rank operator $B$,
%% but also for $\hatA$-compact operators $B$.

\begin{theorem}[LAP vs. existence of virtual states]
\label{theorem-m}
%% Let $A\in\scrC(\bfX)$,
%% $\varOmega\subset\C\setminus\sigma(A)$,
%% and $\bfE\hookrightarrow\bfX\hookrightarrow\bfF$ (densely and continuously).
Let $A\in\scrC(\bfX)$ and $\bfE\hookrightarrow\bfX\hookrightarrow\bfF$
satisfy Assumption~\ref{ass-virtual}.
Let $\varOmega\subset\C\setminus\sigma(A)$.
Let $z_0\in\sigma\sb{\mathrm{ess}}(A)\cap\p\varOmega$
be of rank $r\in\N_0$ relative to $(\varOmega,\bfE,\bfF)$.
For $B\in\scrQ_{\varOmega,\bfE,\bfF}(A-z_0 I)$ (which is nonempty),
define \emph{the space of virtual states} by
\[
\frakM_{\varOmega,\bfE,\bfF}(A-z_0 I)
:=\big\{
\Psi\in\range\big((A+B-z_0 I)^{-1}_{\varOmega,\bfE,\bfF}\big)
\sothat
(\hatA-z_0 I)\Psi=0
\big\}
\subset\bfF,
\]
where $(A+B-z_0 I)^{-1}_{\varOmega,\bfE,\bfF}:\,\bfE\to\bfF$.
Then:
\begin{enumerate}
\item
\label{theorem-m-1}
$\frakM_{\varOmega,\bfE,\bfF}(A-z_0 I)$
does not depend on the choice of
$B\in\scrQ_{\varOmega,\bfE,\bfF}(A-z_0 I)$;
\item
\label{theorem-m-2}
There is the inclusion
$\bfE\cap\ker(A-z_0 I)\subset\frakM_{\varOmega,\bfE,\bfF}(A-z_0 I)$;
\item
\label{theorem-m-3}
%$r:=\min\{\rank B\sothat
%B\in\scrQ_{\varOmega,\bfE,\bfF}(A-z_0 I)\}<\infty$;
$
\dim\frakM_{\varOmega,\bfE,\bfF}(A-z_0 I)=r.
$
\begin{comment}
\item
\label{theorem-m-4}
If $B\in\scrQ_{\varOmega,\bfE,\bfF}(A-z_0 I)$
is of rank $r$,
then the null space and the generalized null space
of $I-B(A+B-z_0 I)^{-1}_{\varOmega,\bfE,\bfF}$
coincide.
Similarly,
the null space and the generalized null space
of $I-(A+B-z_0 I)^{-1}_{\varOmega,\bfE,\bfF} B$
coincide.
\end{comment}
\end{enumerate}
\end{theorem}

%Here is a standard
%application in the context of differential operators.

\begin{example}
\label{example-laplace-1d}
Let $A=-\p_x^2$ in $L^2(\R)$, with $\dom(A)=H^2(\R)$.
We note that its resolvent
$R_0^{(1)}(z)=(A-z I)^{-1}$, $z\in\C\setminus\overline{\R_{+}}$,
with the integral kernel
$R_0^{(1)}(x,y;z)=\fra{e^{-\sqrt{-z}\abs{x-y}}}{(2\sqrt{-z})}$,
$\Re\sqrt{-z}>0$,
does not extend to a linear mapping
$L^2_s(\R)\to L^2_{-s'}(\R)$,
for some particular $s,\,s'\ge 0$,
which would be bounded uniformly for
$z\in\mathbb{D}_\delta\setminus\overline{\R_{+}}$
with some $\delta>0$.
At the same time,
if $V\in C\sb{\mathrm{comp}}([-a,a],\C)$
is any potential
such that the solution
$\theta_{+}(x)$ to $(-\p_x^2+V)\theta=0$,
$\theta\at{x\ge a}=1$,
remains unbounded for $x\le 0$
(one can take
$V\ge 0$ not identically zero),
so that it is linearly independent
with $\theta_{-}(x)$
(solution which equals one for $x< -a$),
then
%%, by Lemma~\ref{lemma-1d},
for any $s,\,s'>1/2$, $s+s'\ge 2$,
the resolvent
$R_V(z)=(A+V-z I)^{-1}$
extends to a bounded linear mapping
$L^2_s(\R)\to L^2_{-s'}(\R)$
for all $z\in\mathbb{D}_\delta\setminus\overline{\R_{+}}$
with some $\delta>0$
and has a limit
in the strong operator topology
as $z\to z_0=0$, $z\not\in\overline{\R_{+}}$;
thus, $z_0=0$ is a regular point of $A+V$
relative to $\C\setminus\overline{\R_{+}}$.
Since the operator of multiplication by $V(x)$
is $A$-compact,
$z_0=0$ is a virtual level of $A=-\p_x^2$ in $L^2(\R)$
(relative to $\C\setminus\overline{\R_{+}}$).
\end{example}

\begin{definition}[Genuine virtual levels]
If
$\frakM_{\varOmega,\bfE,\bfF}(A-z_0 I)\not\subset\bfX$,
then we say that $z_0$ is a
\emph{genuine virtual level}
of $A$
relative to $\varOmega$,
and call
any $\Psi\in\frakM_{\varOmega,\bfE,\bfF}(A-z_0 I)\setminus\bfX$
a \emph{virtual state of $A$
corresponding to $z_0$ relative to $\varOmega$}.
A virtual level can be both an eigenvalue
and a genuine virtual level,
with a corresponding eigenfunction $\psi\in\bfX$
and a virtual state
$\Psi\in\frakM_{\varOmega,\bfE,\bfF}(A-z_0 I)\setminus\bfX$.
\end{definition}

\begin{theorem}[LAP vs. bifurcations]
\label{theorem-b}
%% Let $A\in\scrC(\bfX)$,
%% $\varOmega\subset\C\setminus\sigma(A)$,
%% and $\bfE\hookrightarrow\bfX\hookrightarrow\bfF$ (densely and continuously).
Let $A\in\scrC(\bfX)$ and $\bfE\hookrightarrow\bfX\hookrightarrow\bfF$
satisfy Assumption~\ref{ass-virtual}.
Let $\varOmega\subset\C\setminus\sigma(A)$.
Assume that $z_0\in\sigma\sb{\mathrm{ess}}(A)\cap\p\varOmega$.
\begin{enumerate}
\item
\label{theorem-b-0}
If there is a sequence of perturbations
$
V_j\in\scrB(\bfF,\bfE)$,
$
\lim\sb{j\to\infty}\norm{V_j}_{\bfF\to\bfE}=0$,
and a sequence of eigenvalues
$
z_j\in\sigma\sb{\mathrm{d}}(A+V_j)\cap\varOmega$,
$z_j\to z_0$,
then there is no limit
$\wlim\limits\sb{z\to z_0,\,z\in\varOmega}(A-z I)^{-1}$
%%of $(A-z I)^{-1}$ as $z\to z_0$, $z\in\varOmega$,
in the weak operator topology
of mappings $\bfE\to\bfF$.
%% \nb{I guess you already asked the question:
%%  Can we say that if there are say $r$ such distinct sequences $(z_j^{(a)})_j$ then $z_0$ is at least of rank $r$?}
%% \ac{No, we cannot prove this...
%% Perhaps we could do this if there is the analytic continuation of the resolvent
%% into an open neighborhood of $z_0$ (like in \cite{agmon1998perturbation}).
%% }
\item
\label{theorem-b-1}
Assume that
$z_0$ is a virtual level of $A$
of finite rank relative to $(\varOmega,\bfE,\bfF)$,
and moreover assume that there is
$\delta>0$
and $B\in\scrB_{00}(\bfF,\bfE)$
%% of rank $\rank B=r$
such that there is a limit
\[
(A+B-z I)^{-1}_{\varOmega,\bfE,\bfF}
:=
\slim\limits\sb{z\to z_0,\,z\in\varOmega\cap\mathbb{D}_\delta(z_0)}(A+B-z I)^{-1}
\]
in the \emph{strong operator topology} of mappings
$\bfE\to\bfF$.
There is $\delta_1\in(0,\delta)$ such that for any sequence
$z_j\in\varOmega\cap\mathbb{D}_{\delta_1}(z_0)$,
$z_j\to z_0$,
there is a sequence
\[
V_j\in\scrB_{00}(\bfF,\bfE),
\qquad
\norm{V_j}\sb{\bfF\to\bfE}\to 0,
\qquad
z_j\in\sigma\sb{\mathrm{d}}(A+V_j),
\quad
j\in\N.
\]
%% Then for any $r$ sequences
%% $\big(z_j^{(a)}\in\varOmega\big)_{j\in\N}$,
%% $z_j^{(a)}\to z_0$,
%% $1\le a\le r$,
%% there is a sequence
%% $V_j\in\scrB_{00}(\bfF,\bfE)$
%% such that
%% $\norm{V_j}\sb{\bfF\to\bfE}\to 0$
%% and such that
%% $z_j^{(a)}\in\sigma\sb{\mathrm{d}}(A+V_j)$,
%% $j\in\N$,
%% $1\le a\le r$.
%%
%% (As the matter of fact,
%% $V_j\in\scrB_{00}(\bfF,\bfE)$
%% and
%% $\norm{V_j}\sb{\bfF\to\bfE}\to 0$.)
%% \item
%% \label{theorem-b-3}
%% \ac{DREAM:}
%% If there is a sequence
%% $V_j\in\scrB(\bfF,\bfE)$
%% and
%% $n$ sequences
%% $z_j^{(a)}\in\sigma\sb{\mathrm{d}}(A+V_j)\cap\varOmega$,
%% $j\in\N$,
%% $z_j^{(a)}\to z_0$,
%% $1\le a\le n$,
%% such that
%% \[
%% (A+B-z_j I)^{-1}V_j:\,\bfF\to\bfF
%% \]
%% converge to zero
%% in the uniform operator topology as $j\to\infty$,
%% then
%% $z_0$ is a virtual level of $A$ relative to $\varOmega$ of rank
%% $r\ge n$.
\end{enumerate}
\end{theorem}

\begin{example}
[Virtual levels of $-\Delta+V$ at $z_0\ge 0$]
\label{example-3d-bifurcations}
For $x\in\R^3$ and $\zeta\in\overline{\C\sb{+}}$, define
\[
\psi(x,\zeta)=
\begin{cases}
\fra{e^{\jj \zeta\abs{x} }}{\abs{x}},&\abs{x}\ge 1,
\\
(\fra{(3-\abs{x}^2)}{2})e^{\jj \zeta(1+\abs{x}^2)/2},&0\le \abs{x}<1,
\end{cases}
\quad
\psi(\cdot,\zeta)\in C^2(\R^3),
\]
so that
$-\Delta\psi=\zeta^2 \psi$
for $x\in\R^3\setminus\mathbb{B}^3_1$.
Since $\psi(x,\zeta)\ne 0$
for all $x\in\R^3$ and $\zeta\in\overline{\C\sb{+}}$,
we may define the potential
$V(x,\zeta)$ by the relation
$-\Delta\psi+V\psi=\zeta^2\psi$.
Then, for each $\zeta\in\overline{\C\sb{+}}$,
$V(\cdot,\zeta)\in L^\infty(\R^3)$
is
%% spherically symmetric,
piecewise smooth,
with
$\supp V\subset\mathbb{B}^3_1$.
For
$\zeta\in\C\sb{+}$,
one has $z=\zeta^2\in\sigma\sb{\mathrm{p}}(-\Delta+V(\zeta))$,
so
for each $\zeta_0\ge 0$
there is an eigenvalue family bifurcating from
$z_0=\zeta_0^2\in\sigma\sb{\mathrm{ess}}(-\Delta+V(\cdot,\zeta_0))
=\overline{\R_{+}}$ into $\C_{+}$.
By Theorem~\ref{theorem-b}~\itref{theorem-b-0},
$z_0=\zeta_0^2$
is a virtual level of
$-\Delta+V(x,\zeta_0)$.
\end{example}

\begin{lemma}[The Fredholm alternative]
%% Let $A\in\scrC(\bfX)$,
%% $\varOmega\subset\C\setminus\sigma(A)$,
%% and $\bfE\hookrightarrow\bfX\hookrightarrow\bfF$ (densely and continuously).
Let $A\in\scrC(\bfX)$ and $\bfE\hookrightarrow\bfX\hookrightarrow\bfF$
satisfy Assumption~\ref{ass-virtual}.
Let $\varOmega\subset\C\setminus\sigma(A)$.
Assume that $z_0\in\sigma\sb{\mathrm{ess}}(A)\cap\p\varOmega$
is of rank $r\in\N_0$ relative to
$(\varOmega,\bfE,\bfF)$.
%Assume that $\bar z_0\in\sigma\sb{\mathrm{ess}}(A^*)$
%is of the same rank $r$
%relative to $(\varOmega^*,\,\bfF^*,\,\bfE^*)$.
Then there is a projector
$P\in\End(\bfE)$,
with $\rank P=r$, such that
for any $B\in\scrQ_{\varOmega,\bfE,\bfF}(A-z_0 I)$
the problem
\[
%begin{eqnarray}\label{the-problem}
(\hatA-z_0 I)u=f,
\qquad
f\in\bfE,
\qquad
u\in\range\big((A+B-z_0)^{-1}_{\varOmega,\bfE,\bfF}\big)\subset\bfF,
%\end{eqnarray}
\]
has a solution
%%$u\in\frakM_{\varOmega,\bfE,\bfF}(A-z_0 I)\subset\bfF$
if and only if $Pf=0$.
%% This relation is equivalent to
%% $\langle\theta_j,f\rangle=0$, $1\le j\le r$,
%% with $\theta_j\in\frakM_{\varOmega^*,\bfF^*,\bfE^*}(A^*-\bar z_0 I)\subset\bfE^*$
%% virtual states
%% corresponding to the virtual level $\bar z_0\in\sigma\sb{\mathrm{ess}}(A^*)$.
This solution
%% to \eqref{the-problem}
is unique under an additional requirement $Q u=0$,
with $Q\in\End(\bfF)$
a projection
onto $\frakM_{\varOmega,\bfE,\bfF}(A-z_0 I)\subset\bfF$.
\end{lemma}

The next result is related to \cite[Proposition 4.1]{agmon1998perturbation}
(see Remark~\ref{remark-roma}).
%and Example~\ref{example-1d}.

\begin{theorem}[Independence from regularizing spaces]
\label{theorem-a}
Let $A\in\scrC(\bfX)$.
Let $\bfE_i$ and $\bfF_i$, $i=1,\,2$,
be Banach spaces with dense continuous embeddings
\[
\bfE_i\mathop{\longhookrightarrow}\limits\sp{\imath_i}
\bfX\mathop{\longhookrightarrow}\limits\sp{\jmath_i}
\bfF_i,
\qquad
i=1,\,2.
\]
Assume that
$\bfE_1$ and $\bfE_2$ are mutually dense, 
in the sense that
$\imath_i^{-1}\big(\imath_1(\bfE_1)\cap\imath_2(\bfE_2)\big)$
are dense in $\bfE_i$,
$i=1,\,2$,
that $\bfF_1,\,\bfF_2$ are continuously embedded into 
a Hausdorff vector space $\bfG$,
with $\jmath_1(x)=\jmath_2(x)$ (as an element of $\bfG$) for each $x\in\bfX$,
and that
there is an extension $\hatA\in\scrC\big(\bfF_1+\bfF_2\big)$ of $A$
with dense domain
$\dom(\hatA)\subset \bfF_1+\bfF_2$.
% \ac{added; could this be weakened??}
%% \nb{\sout{
%%  Assume that
%%  $\jmath_i\circ\imath_i(\bfE_i)\cap\dom(\hatA)$ is dense in $\bfF_i$,
%%  $i=1,\,2$} Changed}
% the set $\jmath_1\big(\{\imath_1(\phi)\in\dom(A),\,(A-z_0 I) \imath_1(\phi)\in \imath_1(\bfE_1)\}\big)$ is dense in 
% $\{\phi\in\dom(\hatA),\,(\hatA-z_0 I)\phi\in \jmath_1\circ\imath_1(\bfE_1)\}$ and 
% the set $\jmath_2\big(\{\imath_2(\phi)\in\dom(A),\,(A-z_0 I) \imath_2(\phi)\in \imath_2(\bfE_2)\}\big)$ is dense in 
% $\{\phi\in\dom(\hatA),\,(\hatA-z_0 I)\phi\in \jmath_2\circ\imath_2(\bfE_2)\}$ 
%For $i\in\{1,2\}$,
Let
\[
\dom(A_{\bfE_i\shortto\bfE_i})
=\{\phi\in\bfE_i\sothat\,
\imath_i(\phi)\in\dom(A),
\,A\imath_i(\phi)\in \imath_i(\bfE_i)\},
\quad
i=1,\,2
\]
and 
\[
\dom(\hatA_{\bfE_i\shortto\bfE_i})
=\{\phi\in\bfE_i\sothat\,
\jmath_i\circ\imath_i(\phi)\in\dom(\hatA),
\,\hatA\jmath_i\circ\imath_i(\phi)\in\jmath_i\circ\imath_i(\bfE_i)\},
\quad
i=1,\,2,
\]
and assume that for $i=1,\,2$
the space
$\jmath_i\circ\imath_i(\dom(A_{\bfE_i\shortto\bfE_i}))$
is dense in the space
$\jmath_i\circ\imath_i(\dom(\hatA_{\bfE_i\shortto\bfE_i}))$
in the topology induced by the graph norm of $\hatA$
% both for the graph norm of $\dom(\hatA)$
(it follows that both triples $A,\,\bfE_1,\,\bfF_1$
and
$A,\,\bfE_2,\,\bfF_2$
satisfy Assumption~\ref{ass-virtual}~\itref{ass-virtual-3}).

Let $\varOmega\subset\C\setminus\sigma(A)$
and assume that
$z_0\in\sigma\sb{\mathrm{ess}}(A)\cap\p\varOmega$
and that
$\scrQ_{\varOmega,\bfE_i,\bfF_i}(A-z_0 I)\ne\emptyset$,
$i=1,\,2$.
Then
$z_0$ is of the same rank relative to
$(\varOmega,\bfE_1,\bfF_1)$
and relative to
$(\varOmega,\bfE_2,\bfF_2)$,
and moreover
\[
\frakM_{\varOmega,\bfE_1,\bfF_1}(A-z_0 I)
=
\frakM_{\varOmega,\bfE_2,\bfF_2}(A-z_0 I).
\]
\end{theorem}

Above,
$
\bfE_1\cap\bfE_2
=\{\phi\in\bfX\sothat
\exists(\phi_1,\phi_2)\in\bfE_1\times\bfE_2,
\ \imath_1(\phi_1)=\imath_2(\phi_2)=\phi
\},
$
with
\[
\norm{\phi}_{\bfE_1\cap\bfE_2}=
\norm{\phi_1}_{\bfE_1}+\norm{\phi_2}_{\bfE_2}
,
\]
and
$
\bfF_1+\bfF_2
=
\{\psi_1+\psi_2\in\bfG\sothat
%\exists \psi_1\in\bfF_1,\ \exists \psi_2\in\bfF_2,
%\exists
(\psi_1,\psi_2)\in\bfF_1\times\bfF_2
%\ \psi=\psi_1+\psi_2
\},
$
with the norm
\[
\norm{\psi}_{\bfF_1+\bfF_2}
=
\inf\limits\sb{
%\substack{
\psi=\psi_1+\psi_2,
%\\
\ (\psi_1,\psi_2)\in\bfF_1\times\bfF_2
%}
}
\big(
\norm{\psi_1}_{\bfF_1}+\norm{\psi_2}_{\bfF_2}
\big).
\]
We point out that
if $\bfE_1$ and $\bfE_2$ are not mutually dense
(or similarly if $\bfF_1$ and $\bfF_2$
are not inside a common vector space $\bfG$),
then there is a nontrivial dependence of the rank
of $z_0\in\sigma\sb{\mathrm{ess}}(A)$
from the choice of regularizing subspaces;
see Example~\ref{example-1d}.

\begin{lemma}[Virtual levels of the adjoint]
\label{lemma-a-2}
%% Let $A\in\scrC(\bfX)$,
%% $\varOmega\subset\C\setminus\sigma(A)$,
%% and $\bfE\hookrightarrow\bfX\hookrightarrow\bfF$ (densely and continuously).
Let $A\in\scrC(\bfX)$ and
%% \nb{changed} 
$\bfE\mathop{\longhookrightarrow}\limits\sp{\imath}\bfX\mathop{\longhookrightarrow}\limits\sp{\jmath}\bfF$
satisfy Assumption~\ref{ass-virtual}.
Moreover, assume that $\bfE$ be reflexive
and let
$A^*:\,\bfX^*\to\bfX^*$ have a closable extension
to a mapping $\widehat{A^*}:\bfE^*\to\bfE^*$
% \ac{added:}
with domain $\dom\big(\widehat{A^*}\big)\subset\bfE^*$.
Let
\[
\mbox{\ac{\sout{$\dom(A^*_{\bfF^*\shortto\bfF^*})
=\{\phi\in \dom(A^*)\cap \jmath^*(\bfF^*),\,A^*\phi\in \jmath^*(\bfF^*)\}$}}}
\]
\[
\dom(A^*_{\bfF^*\shortto\bfF^*})
=\{\theta\in\bfF^*\sothat
\jmath^*(\theta)\in\dom(A^*),
\,A^*\jmath^*(\theta)\in\jmath^*(\bfF^*)\}
\]
and 
\[
\mbox{\ac{\sout{$\dom\big(\widehat{A^*}_{\bfF^*\shortto\bfF^*}\big)
=\{\phi\in \dom(\widehat{A^*})\cap\imath^*\circ\jmath^*(\bfF^*),\,\widehat{A^*}\phi\in\imath^*\circ\jmath^*(\bfF^*)\}$}}}
\]
\[
\dom\big(\widehat{A^*}_{\bfF^*\shortto\bfF^*}\big)
=\{\theta\in\bfF^*\sothat\,
\imath^*\circ\jmath^*(\theta)\in\dom(\widehat{A^*}),
\,\widehat{A^*}\,\imath^*\circ\jmath^*(\theta)\in\imath^*\circ\jmath^*(\bfF^*)\}
\]
and assume that
\ac{\sout{
$\jmath^*\big(\dom\big(A^*_{\bfF^*\shortto\bfF^*}\big)\big)$ is dense in
$\dom\big(\widehat{A^*}_{\bfF^*\shortto\bfF^*}\big)$
}}
$\imath^*\circ\jmath^*\big(\dom\big(A^*_{\bfF^*\shortto\bfF^*}\big)\big)$ is dense in
$\imath^*\circ\jmath^*\big(\dom\big(\widehat{A^*}_{\bfF^*\shortto\bfF^*}\big)\big)$
in the topology induced
by the graph norm of $\widehat{A^*}$
% 
% the set $\imath^*\big(\{\jmath^*(\phi)\in\dom(A^*),\,(A^*-\bar{z_0} I) \jmath^*(\phi)\in \jmath^*(\bfF^*)\}\big)$ is dense in 
% $\{\phi\in\dom_1,\,(A_1-\bar{z_0} I)\phi\in \imath^*\circ\jmath^*(\bfF^*)\}$
% for the graph norm of $\dom(A_1)$
(that is, $A^*$ and
$\bfF^*
\mathop{\longhookrightarrow}\limits\sp{\jmath^*}\bfX^*
\mathop{\longhookrightarrow}\limits\sp{\imath^*}\bfE^*$
satisfy Assumption~\ref{ass-virtual}).

Let $\varOmega\subset\C\setminus\sigma(A)$.
Assume that
$z_0\in\sigma\sb{\mathrm{ess}}(A)\cap\p\varOmega$
is an exceptional point of the essential spectrum of $A$
of rank $r\in\N_0\sqcup\{\infty\}$
relative to $(\varOmega,\bfE,\bfF)$.
Then
$\bar z_0\in\sigma\sb{\mathrm{ess}}(A^*)$
is an exceptional point
of the essential spectrum of $A^*$
of rank $r$
relative to $(\varOmega^*,\,\bfF^*,\,\bfE^*)$,
with
$\varOmega^*:=\big\{\zeta\in\C\sothat\bar\zeta\in\varOmega\big\}$.
\end{lemma}

Above, the assumption
that $\bfE$ is reflexive is needed so that
the existence of a limit of the resolvent
$(A-z I)^{-1}$, $z\to z_0$, $z\in\varOmega$,
in the weak operator topology of mappings
$\bfE\to\bfF$
also provides
the existence of a limit of
$(A^*-\zeta I)^{-1}$, $\zeta\to\bar z_0$, $\bar\zeta\in\varOmega$,
in the weak operator topology of mappings
$\bfF^*\to\bfE^*$.
%% \nb{Can we use $R\subset R^{\ast\ast}$ ?
%% Maybe not in the review, drop this comment if you agree.}

\section{Application to the Schr\"odinger operators}

Let us illustrate how our approach
from Section~\ref{sect-general}
can be applied to the study of properties of virtual states
and LAP estimates
of Schr\"odinger operators
with nonselfadjoint potentials.
According to the developed theory,
it suffices to derive the estimates for model operators.
If $d\ge 3$, one derives optimal LAP estimates
for the Laplacian
(see \cite[Proposition 2.4]{MR0368656} and \cite[\S 3.3]{virtual-levels});
for $d=2$,
one considers the regularization
$-\Delta+\vargamma\unity_{\mathbb{B}^2_1}$, $0<\vargamma\ll 1$,
destroying the virtual level at $z_0=0$
\cite[\S 3.2]{virtual-levels}.
For $d=1$,
one could proceed in the same way as for $d=2$,
although
the estimates can be derived directly for any $V\in L^1_1(\R,\C)$
\cite[\S 3.1]{virtual-levels}.
The resulting estimates will be valid for all complex-valued potentials
when there are no virtual levels.
When there are virtual levels,
then the corresponding virtual states
can be characterized as functions
in the range of the regularized resolvent (Theorem~\ref{theorem-m}).

%For the selfadjoint case, the results were given in
%Many of such estimates are well-known;
%% We mention that uniform resolvent bounds for
%% selfadjoint
%% Schr\"odinger operators in higher dimensions
%% were studied, for instance, in
%% \cite{kenig1987uniform,gutierrez2004nontrivial,frank2011eigenvalue,frank2017eigenvalue,bouclet2018uniform,ren2018endpoint,mizutani2019eigenvalue},  \cite{kwon2020sharp}.
%% We also mention the classical Rollnik bound:
%% see, for instance,
%% \cite[Example~3, p.~150]{MR0493421}, \cite[Sect.~I.4]{MR0455975},
%% \cite[Proposition~7.1.16]{MR2598115}.
%% %%(for a discrete analog in this context see \cite{MR4009459}).

\begin{theorem}\label{theorem-sh}
Assume that $V\in L^\infty_\rho(\R^d,\C)$,
$\rho>2$;
if $d=1$, it suffices to have $V\in L^1_1(\R^1,\C)$
\textup(see \eqref{def-lp}\textup).
Let $A=-\Delta+V$
in $L^2(\R^d)$, $\dom(A)=H^2(\R^d)$,
$d\ge 1$.
\begin{itemize}
\item[$\bullet$]
If $z_0=0$ is not a virtual level of $A$
relative to $\varOmega=\C\setminus\overline{\R_{+}}$,
$\bfE=L^2_s(\R^d)$, $\bfF=L^2_{-s'}(\R^d)$,
with $s,\,s'>0$ sufficiently large,
then the following mappings are continuous:
\[
(A-z_0 I)^{-1}_\varOmega:
\:
L^2_s(\R^d)\to L^2_{-s'}(\R^d),
\qquad
\begin{cases}
s+s'\ge 2,\ s,\,s'>\fra 1 2,
&d=1;
\\
s,\,s'>2-\fra d 2,\ s,\,s'\ge 0,
&d\ge 2.
\end{cases}
\]
%% with $s+s'\ge 2$, $s,\,s'>\fra 1 2
%% \mbox{ \ if \ }d=1;
%% $
%% \ \ $s,\,s'>2-\fra d 2,\ s,\,s'\ge 0
%% \mbox{ \ if \ }d\ge 2.
%% $
Moreover, for $1\le d\le 3$,
\[
(A-z_0 I)^{-1}_\varOmega:
\:
L^1(\R^d)\to L^2_{-s}(\R^d),\quad
L^2_s(\R^d)\to L^\infty(\R^d),
\quad
\forall s>2-d/2.
%\quad 1\le d\le 3.
\]
\item[$\bullet$]
If $z_0=0$ is a virtual level of $A$,
then there is a nonzero solution to the following problem:
%which satisfies
\[
(A-z_0 I)\Psi=0,
\qquad
\Psi\in
\begin{cases}
L^\infty(\R^d),&d\le 2;
\\
L^\infty(\R^d)\cap L^2_{-1/2-0}(\R^d),&d=3;
\\
L^2_{-0}(\R^d),&d=4;
\\
L^2(\R^d),&d\ge 5.
\end{cases}
\]
\end{itemize}
\end{theorem}

For more details and references, see \cite{virtual-levels}.
Related results on properties of eigenstates
and virtual states
are in
\cite{gesztesy2020absence}
(Schr\"odinger and massive Dirac operators in dimension $d\ge 3$
and massless Dirac operators in $d\ge 2$)
and in \cite[Theorem 2.3]{barth2020efimov}
(Schr\"odinger operators in $d\le 2$).
Let us note that,
prior to \cite{virtual-levels},
the nonselfadjoint case has not been considered
(although some results appeared in \cite{cuccagna2005bifurcations}).
Moreover, as far as we know,
even in the selfadjoint case,
the LAP in dimension $d=2$
at the threshold
when it is a regular point
of the essential spectrum
was not available.
Although the 
$L^1\to L^2_{-s}$ and $L^2_s\to L^\infty$ estimates
stated above are straightforward
in dimension $d=1$ and $d=3$,
we also do not have a reference.

%% According to the general theory developed in
%% Section~\ref{sect-general},
%% the main ingredient in proving the above theorem
%% is the proof of the LAP estimates for the
%% perturbed (``regularized'')
%% Laplacian
%% $A=-\Delta+\vargamma\unity_{\mathbb{B}^d_1}$,
%% with some $0<\vargamma\ll 1$,
%% in dimension $d\le 2$
%% (derived in \cite[Sections 3.1, 3.2]{virtual-levels})
%% and for the free Laplacian $A=-\Delta$ in dimension $d\ge 3$
%% (see e.g. \cite[Section 3.3]{virtual-levels}).

% \ac{NEW (July 2021)}

\begin{remark}
According to
Theorem~\ref{theorem-a},
the absence of uniform estimates of the form
$(-\Delta-z I)^{-1}:\,L^p(\R^d)\to L^q(\R^d)$
for $z\in\C\setminus\overline{\R_{+}}$
for $d\le 2$ \cite{kwon2020sharp}
is directly related to the fact that
there is a virtual level of $-\Delta$
at $z_0=0$ in dimensions $d\le 2$
relative to
$\big(\C\setminus\overline{\R_{+}},\,
L^2_s(\R^d)$, $L^2_{-s'}(\R^d)\big)$,
with arbitrarily large $s,\,s'\ge 0$.
\end{remark}

% \ac{NEW (July 2021)}

\begin{example}
Since $\Psi\equiv 1$ is an $L^\infty$-solution to
$\p_x^2 u=0$,
by Theorem~\ref{theorem-sh}, $z_0=0$ is not a regular point
of the essential spectrum of the Laplacian in $\R$
relative to
$\big(\C\setminus\overline{\R_{+}},\,L^2_s(\R),\,L^2_{-s'}(\R)\big)$,
with $s,\,s'>1/2$, $s+s'\ge 2$.

Now let us show that $z_0=0$ is a virtual level of rank $r=1$
(relative to the same triple
$\big(\C\setminus\overline{\R_{+}},\,L^2_s(\R),\,L^2_{-s'}(\R)\big)$).
Consider a rank one perturbation of the Laplacian,
\[
A=-\p_x^2+\unity_{[-1,1]}\otimes
\langle\unity_{[-1,1]},\cdot\rangle,
\qquad
A\in\scrC(L^2(\R)),
\qquad
\dom(A)=H^2(\R),
\]
with $\unity_{[-1,1]}$ the characteristic function
of the interval $[-1,1]$.
We claim that $z_0=0$ is a regular point of
$\sigma\sb{\mathrm{ess}}(A)$.
%which will imply that $z_0=0$
%is a virtual level of rank one of the Laplacian $-\p_x^2$.
%In dimension $d=1$, 
Indeed, the relation $A u=0$ takes the form
\begin{eqnarray}\label{hu-1}
u''(x)=c\unity_{[-1,1]}(x),
\qquad
x\in\R,
\qquad
c:=\int_{-1}^1u(y)\,dy.
\end{eqnarray}
The requirement $u\in L^\infty(\R)$ implies that
$u(x)=a_{-}$ for $x<-1$ and $u(x)=a_{+}$ for $x>-1$,
with some $a\sb\pm\in\C$;
for $-1<x<1$, one has $u=a+bx+c x^2/2$, 
with some $a,\,b\in\C$.
The continuity of the first derivative at $x=\pm 1$
leads to $b-c=0$ and $b+c=0$, hence $b=c=0$;
at the same time, the relation
$0=c=\int_{-1}^1 a\,dx$ implies that $a=0$ and thus
$u(x)$ is identically zero.
Hence, there is no nontrivial $L^\infty$-solution to \eqref{hu-1}.
By Theorem~\ref{theorem-sh},
$z_0=0$ is a regular point of $\sigma\sb{\mathrm{ess}}(A)$,
hence it is a virtual level
of rank one of $-\p_x^2$.
\end{example}

\begin{acknowledgements}
The authors are most grateful to
Gregory Berkolaiko,
Kirill Cherednichenko,
Fritz Gesztesy,
Bill Johnson,
Alexander V. Kiselev,
Mark Malamud,
Alexander Nazarov,
Yehuda Pinchover,
Roman Romanov,
Thomas Schlump\-recht,
Vladimir Sloushch,
Tatiana Suslina,
Cyril Tintarev,
Boris Vainberg,
and Dmitrii Yafaev
for their attention and advice.
The authors are indebted to the anonymous referee
for bringing to their attention
several important references.
\end{acknowledgements}

% Authors must disclose all relationships or interests that 
% could have direct or potential influence or impart bias on 
% the work: 
%
%\section*{Conflict of interest}

\noindent
{\small
{\bf Conflict of interest.\ }
The authors declare that they have no conflict of interest.
}

% BibTeX users please use one of
%\bibliographystyle{spbasic}      % basic style, author-year citations
%\bibliographystyle{spmpsci}      % mathematics and physical sciences
%\bibliographystyle{spphys}       % APS-like style for physics

\bibliographystyle{sima-doi}
\bibliography{bibcomech}   % name your BibTeX data base

\def\cprime{\mbox{'}} \def\polhk#1{\setbox0=\hbox{#1}{\ooalign{\hidewidth
  \lower1.5ex\hbox{`}\hidewidth\crcr\unhbox0}}}
\begin{thebibliography}{BGW85}

\bibitem[AF36]{PhysRev.50.899}
E.~Amaldi and E.~Fermi,
  \href{https://link.aps.org/doi/10.1103/PhysRev.50.899}{{\em On the absorption
  and the diffusion of slow neutrons\/}}, Phys. Rev. {\bf 50} (1936), pp.
  899--928.

\bibitem[AG81]{akhiezer1981theory}
N.~I. Akhiezer and I.~M. Glazman, {\em Theory of linear operators in {H}ilbert
  space (volumes {I} and {II})\/}, London and Scottish Academic Press,
  Edinburgh, 1981.

\bibitem[Agm70]{agmon1970spectral}
S.~Agmon,
  \href{https://www.mathunion.org/fileadmin/ICM/Proceedings/ICM1970.2/ICM1970.2.ocr.pdf}{{\em
  Spectral properties of {S}chr{\"o}dinger operators\/}},
  \href{https://www.mathunion.org/fileadmin/ICM/Proceedings/ICM1970.2/ICM1970.2.ocr.pdf}{in
  \href{https://www.mathunion.org/fileadmin/ICM/Proceedings/ICM1970.2/ICM1970.2.ocr.pdf}{{\em
  Actes, {C}ongr{\`e}s intern. {M}ath.\/}}}, vol.~2, pp. 679--683, 1970.

\bibitem[Agm75]{agmon1975spectral}
S.~Agmon,
  \href{http://archive.numdam.org/ARCHIVE/ASNSP/ASNSP_1975_4_2_2/ASNSP_1975_4_2_2_151_0/ASNSP_1975_4_2_2_151_0.pdf}{{\em
  Spectral properties of {S}chr{\"o}dinger operators and scattering theory\/}},
  Ann. Scuola Norm. Sup. Pisa Cl. Sci. (4) {\bf 2} (1975), pp. 151--218.

\bibitem[Agm82]{agmon1982positivity}
S.~Agmon, {\em On positivity and decay of solutions of second order elliptic
  equations on {R}iemannian manifolds\/}, in {\em Methods of Functional
  Analysis and Theory of Elliptic Equations\/}, pp. 19--52, Liguori Editore,
  Naples, 1982.

\bibitem[Agm98]{agmon1998perturbation}
S.~Agmon,
  \href{http://dx.doi.org/10.1002/(SICI)1097-0312(199811/12)51:11/12<1255::AID-CPA2>3.0.CO;2-O}{{\em
  A perturbation theory of resonances\/}}, Communications on Pure and Applied
  Mathematics {\bf 51} (1998), pp. 1255--1309.

\bibitem[AN18]{apushkinskaya2018cherish}
D.~Apushkinskaya and A.~Nazarov,
  \href{http://www.mathsoc.spb.ru/pantheon/smirnov/footprints.pdf}{{\em
  ``{C}herish the footprints of {M}an on the sand of {T}ime!'' ({V}.{I}.
  {S}mirnov)\/}}, Algebra i Analiz {\bf 30} (2018), pp. 3--17.

\bibitem[BAD87]{benartzi1987limiting}
M.~Ben-Artzi and A.~Devinatz, \href{http://dx.doi.org/10.1090/memo/0364}{{\em
  The limiting absorption principle for partial differential operators\/}},
  Mem. Amer. Math. Soc. {\bf 66} (1987), pp. iv+70.

\bibitem[BBV20]{barth2020efimov}
S.~Barth, A.~Bitter, and S.~Vugalter,
  \href{http://arxiv.org/abs/2010.08452}{{\em On the {E}fimov effect in systems
  of one- or two-dimensional particles\/}} (2020), \eprint{2010.08452}.

\bibitem[BC19]{opus}
N.~Boussa{\"\i}d and A.~Comech, \href{http://dx.doi.org/10.1090/surv/244}{{\em
  Nonlinear {D}irac equation. {S}pectral stability of solitary waves\/}}, vol.
  244 of {\em Mathematical Surveys and Monographs\/}, American Mathematical
  Society, Providence, RI, 2019.

\bibitem[BC21]{virtual-levels}
N.~Boussa{\"\i}d and A.~Comech, \href{https://arxiv.org/abs/2101.11979}{{\em
  Virtual levels and virtual states of linear operators in {B}anach spaces.
  {A}pplications to {S}chr\"{o}dinger operators\/}}  (2021),
  \eprint{2101.11979}.

\bibitem[Ber57]{berezanskii1957eigenfunction}
Y.~M. Berezanskii, \href{http://mi.mathnet.ru/eng/msb/v85/i1/p75}{{\em
  Eigenfunction expansions of self-adjoint operators\/}}, Mat. Sb. (N.S.) {\bf
  43 (85)} (1957), pp. 75--126.

\bibitem[BGD88]{MR952661}
D.~Boll{\'e}, F.~Gesztesy, and C.~Danneels,
  \href{http://www.numdam.org/item/AIHPA_1988__48_2_175_0}{{\em Threshold
  scattering in two dimensions\/}}, Ann. Inst. H. Poincar\'e Phys. Th\'eor.
  {\bf 48} (1988), pp. 175--204.

\bibitem[BGK87]{MR877834}
D.~Boll{\'e}, F.~Gesztesy, and M.~Klaus,
  \href{http://dx.doi.org/10.1016/0022-247X(87)90281-2}{{\em Scattering theory
  for one-dimensional systems with {$\int dx\,V(x)=0$}\/}}, J. Math. Anal.
  Appl. {\bf 122} (1987), pp. 496--518.

\bibitem[BGW85]{MR768299}
D.~Boll{\'e}, F.~Gesztesy, and S.~F.~J. Wilk,
  \href{http://jot.theta.ro/jot/archive/1985-013-001/1985-013-001-001.pdf}{{\em
  A complete treatment of low-energy scattering in one dimension\/}}, J.
  Operator Theory {\bf 13} (1985), pp. 3--31.

\bibitem[Bir56]{birman1956theory}
M.~S. Birman, \href{http://mi.mathnet.ru/eng/msb/v80/i4/p431}{{\em On the
  theory of self-adjoint extensions of positive definite operators\/}},
  Matematicheskii Sbornik {\bf 80} (1956), pp. 431--450.

\bibitem[Bir61]{birman1961spectrum}
M.~S. Birman, \href{http://mi.mathnet.ru/eng/msb/v97/i2/p125}{{\em On the
  spectrum of singular boundary-value problems\/}}, Mat. Sb. (N.S.) {\bf 55
  (97)} (1961), pp. 125--174.

\bibitem[Bou06]{boussaid2006stable}
N.~Boussa{\"\i}d, \href{http://dx.doi.org/10.1007/s00220-006-0112-3}{{\em
  Stable directions for small nonlinear {D}irac standing waves\/}}, Comm. Math.
  Phys. {\bf 268} (2006), pp. 757--817.

\bibitem[Bou08]{boussaid2008asymptotic}
N.~Boussa{\"\i}d, \href{http://dx.doi.org/10.1137/070684641}{{\em On the
  asymptotic stability of small nonlinear {D}irac standing waves in a resonant
  case\/}}, SIAM J. Math. Anal. {\bf 40} (2008), pp. 1621--1670.

\bibitem[Bro61]{browder1961spectral}
F.~E. Browder, \href{http://dx.doi.org/10.1007/BF01343363}{{\em On the spectral
  theory of elliptic differential operators. {I}\/}}, Mathematische Annalen
  {\bf 142} (1961), pp. 22--130.

\bibitem[Car34]{carleman1934theorie}
T.~Carleman, {\em Sur la th{\'e}orie math{\'e}matique de l'{\'e}quation de
  {S}chr{\"o}dinger\/}, Almqvist \& Wiksell, 1934.

\bibitem[CFG86]{chiarenza1986harnack}
F.~Chiarenza, E.~Fabes, and N.~Garofalo,
  \href{http://dx.doi.org/10.2307/2046194}{{\em Harnack's inequality for
  {S}chr{\"o}dinger operators and the continuity of solutions\/}}, Proceedings
  of the American Mathematical Society {\bf 98} (1986), pp. 415--425.

\bibitem[CP05]{cuccagna2005bifurcations}
S.~Cuccagna and D.~Pelinovsky, \href{http://dx.doi.org/10.1063/1.1901345}{{\em
  Bifurcations from the endpoints of the essential spectrum in the linearized
  nonlinear {S}chr\"odinger problem\/}}, J. Math. Phys. {\bf 46} (2005), pp.
  053520, 15.

\bibitem[Dev14]{devyver2014spectral}
B.~Devyver, \href{http://dx.doi.org/10.1016/j.matpur.2014.02.007}{{\em A
  spectral result for {H}ardy inequalities\/}}, Journal de Math\'{e}matiques
  Pures et Appliqu\'{e}es {\bf 102} (2014), pp. 813--853.

\bibitem[EG17]{erdogan2017dirac}
M.~B. Erdo{\u{g}}an and W.~R. Green,
  \href{http://dx.doi.org/10.1007/s00220-016-2811-8}{{\em The {D}irac equation
  in two dimensions: dispersive estimates and classification of threshold
  obstructions\/}}, Comm. Math. Phys. {\bf 352} (2017), pp. 719--757.

\bibitem[EGT19]{erdogan2019dispersive}
M.~B. Erdo{\u{g}}an, W.~R. Green, and E.~Toprak,
  \href{http://dx.doi.org/10.1353/ajm.2019.0031}{{\em Dispersive estimates for
  {D}irac operators in dimension three with obstructions at threshold
  energies\/}}, American Journal of Mathematics {\bf 141} (2019), pp.
  1217--1258.

\bibitem[Eid62]{eidus1962principle}
D.~M. Eidus, \href{http://mi.mathnet.ru/eng/msb/v99/i1/p13}{{\em The principle
  of limit absorption\/}}, Math. Sb. {\bf 57} (1962), pp. 13--44.

\bibitem[Eid69]{eidus1969principle}
D.~M. Eidus, \href{http://mi.mathnet.ru/umn5497}{{\em The principle of limit
  amplitude\/}}, Uspekhi Mat. Nauk {\bf 24} (1969), pp. 91--156.

\bibitem[ES04]{erdogan2004dispersive}
M.~B. Erdo{\u{g}}an and W.~Schlag,
  \href{http://dx.doi.org/10.4310/DPDE.2004.v1.n4.a1}{{\em Dispersive estimates
  for {S}chr\"{o}dinger operators in the presence of a resonance and/or an
  eigenvalue at zero energy in dimension three: {I}\/}}, Dyn. Partial Differ.
  Equ. {\bf 1} (2004), pp. 359--379.

\bibitem[Fad63]{MR0163695}
L.~D. Faddeev, \href{http://mi.mathnet.ru/tm1765}{{\em Mathematical questions
  in the quantum theory of scattering for a system of three particles\/}},
  Trudy Mat. Inst. Steklov. {\bf 69} (1963), p. 122.

\bibitem[Fer35]{fermi1935recombination}
E.~Fermi, \href{http://dx.doi.org/10.1103/PhysRev.48.570}{{\em On the
  recombination of neutrons and protons\/}}, Physical Review {\bf 48} (1935),
  p. 570.

\bibitem[GK55]{gelfand1955eigenfunction}
I.~Gelfand and A.~Kostyuchenko, {\em Eigenfunction expansions of differential
  and other operators\/}, Dokl. Akad. Nauk {\bf 103} (1955), pp. 349--352.

\bibitem[GM74]{MR0368656}
J.~Ginibre and M.~Moulin,
  \href{http://www.numdam.org/item/AIHPA_1974__21_2_97_0/}{{\em Hilbert space
  approach to the quantum mechanical three-body problem\/}}, Ann. Inst. H.
  Poincar\'e Sect. A (N.S.) {\bf 21} (1974), pp. 97--145.

\bibitem[GN20]{gesztesy2020absence}
F.~Gesztesy and R.~Nichols, \href{http://dx.doi.org/10.3934/dcdss.2020243}{{\em
  On absence of threshold resonances for {S}chr{\"o}dinger and {D}irac
  operators\/}}, Discrete Contin. Dyn. Syst. S {\bf 13} (2020), pp. 3427--3460.

\bibitem[GT83]{MR737190}
D.~Gilbarg and N.~S. Trudinger,
  \href{http://dx.doi.org/10.1007/978-3-642-61798-0}{{\em Elliptic partial
  differential equations of second order\/}}, vol. 224 of {\em Grundlehren der
  Mathematischen Wissenschaften\/}, Springer-Verlag, Berlin, 1983, second edn.

\bibitem[Gus09]{guseinov2009concept}
G.~S. Guseinov, \href{http://dx.doi.org/10.1007/s12043-009-0111-y}{{\em On the
  concept of spectral singularities\/}}, Pramana -- J. Phys. {\bf 73} (2009),
  pp. 587--603.

\bibitem[GV61]{gelfand1961some}
I.~M. Gelfand and N.~Y. Vilenkin, {\em Some problems of harmonic analysis.
  {R}igged {H}ilbert Spaces\/}, Fizmatgiz, Moscow, 1961.

\bibitem[GZ91]{gesztesy1991critical}
F.~Gesztesy and Z.~Zhao,
  \href{http://dx.doi.org/10.1016/0022-1236(91)90081-F}{{\em On critical and
  subcritical {S}turm--{L}iouville operators\/}}, J. Functional Analysis {\bf
  98} (1991), pp. 311--345.

\bibitem[Hil47]{hille1947jacob}
E.~Hille, \href{http://dx.doi.org/10.1090/S0002-9904-1947-08789-9}{{\em Jacob
  {D}avid {T}amarkin -- {H}is life and work\/}}, Bull. Amer. Math. Soc. {\bf
  53} (1947), pp. 440--457.

\bibitem[How74]{howland1974poiseux}
J.~S. Howland, \href{http://dx.doi.org/10.2140/pjm.1974.55.157}{{\em Puiseux
  series for resonances at an embedded eigenvalue\/}}, Pacific J. Math. {\bf
  55} (1974), pp. 157--176.

\bibitem[Ign05]{ignatowsky1905reflexion}
W.~Ignatowsky,
  \href{https://gallica.bnf.fr/ark:/12148/bpt6k15325j/f502.image}{{\em
  Reflexion elektromagnetischer {W}ellen an einem {D}raft\/}}, Ann. Phys. {\bf
  18} (1905), pp. 495--522.

\bibitem[Jen80]{MR563367}
A.~Jensen,
  \href{http://projecteuclid.org/getRecord?id=euclid.dmj/1077313862}{{\em
  Spectral properties of {S}chr{\"o}dinger operators and time-decay of the wave
  functions results in ${L}^{2}(\mathbf{R}^{m})$, $m\geq 5$\/}}, Duke Math. J.
  {\bf 47} (1980), pp. 57--80.

\bibitem[Jen84]{MR748579}
A.~Jensen, \href{http://dx.doi.org/10.1016/0022-247X(84)90110-0}{{\em Spectral
  properties of {S}chr{\"o}dinger operators and time-decay of the wave
  functions. {R}esults in ${L}^2(\mathbf{R}^{4})$\/}}, J. Math. Anal. Appl.
  {\bf 101} (1984), pp. 397--422.

\bibitem[JK79]{MR544248}
A.~Jensen and T.~Kato,
  \href{http://projecteuclid.org/getRecord?id=euclid.dmj/1077313577}{{\em
  Spectral properties of {S}chr{\"o}dinger operators and time-decay of the wave
  functions\/}}, Duke Math. J. {\bf 46} (1979), pp. 583--611.

\bibitem[JN01]{MR1841744}
A.~Jensen and G.~Nenciu,
  \href{http://dx.doi.org/10.1142/S0129055X01000843}{{\em A unified approach to
  resolvent expansions at thresholds\/}}, Rev. Math. Phys. {\bf 13} (2001), pp.
  717--754.

\bibitem[KL20]{kwon2020sharp}
Y.~Kwon and S.~Lee, \href{http://dx.doi.org/10.1007/s00220-019-03536-y}{{\em
  Sharp resolvent estimates outside of the uniform boundedness range\/}},
  Communications in Mathematical Physics {\bf 374} (2020), pp. 1417--1467.

\bibitem[KLV19]{konotop2019designing}
V.~V. Konotop, E.~Lakshtanov, and B.~Vainberg,
  \href{https://link.aps.org/doi/10.1103/PhysRevA.99.043838}{{\em Designing
  lasing and perfectly absorbing potentials\/}}, Phys. Rev. A {\bf 99} (2019),
  p. 043838.

\bibitem[Kre46]{krein1946general}
M.~Krein, {\em On a general method of decomposing {H}ermite-positive nuclei
  into elementary products\/}, Dokl. Akad. Nauk {\bf 53} (1946), pp. 3--6.

\bibitem[Kre47]{krein1947theory}
M.~Krein, \href{http://mi.mathnet.ru/eng/msb/v62/i3/p431}{{\em The theory of
  self-adjoint extensions of semi-bounded {H}ermitian transformations and its
  applications. {I}\/}}, Matematicheskii Sbornik {\bf 62} (1947), pp. 431--495.

\bibitem[Kre48]{krein1948hermitian}
M.~Krein, {\em On hermitian operators with directed functionals\/}, Akad. Nauk
  Ukrain. RSR. Zbirnik Prac Inst. Mat {\bf 10} (1948), pp. 83--106.

\bibitem[Kur78]{MR528757}
S.~T. Kuroda, {\em An introduction to scattering theory\/}, vol.~51 of {\em
  Lecture Notes Series\/}, Aarhus Universitet, Matematisk Institut, Aarhus,
  1978.

\bibitem[Lan75]{landau1975szego}
H.~Landau, \href{http://dx.doi.org/10.1007/BF02786820}{{\em On {S}zeg{\"o}'s
  eingenvalue distribution theorem and non-{H}ermitian kernels\/}}, Journal
  d'Analyse Math{\'e}matique {\bf 28} (1975), pp. 335--357.

\bibitem[Lja67]{lyantse1967differential}
V.~Ljance, \href{http://dx.doi.org/10.1090/trans2/60}{{\em On differential
  operators with spectral singularities\/}},
  \href{http://dx.doi.org/10.1090/trans2/60}{in
  \href{http://dx.doi.org/10.1090/trans2/60}{{\em Seven Papers on
  Analysis\/}}}, vol.~60 of {\em Amer. Math. Soc. Transl. Ser. 2\/}, pp.
  185--225, Amer. Math. Soc., Providence, RI, 1967.

\bibitem[LP18]{lucia2018criticality}
M.~Lucia and S.~Prashanth,
  \href{http://dx.doi.org/10.1016/j.jde.2018.05.006}{{\em Criticality theory
  for {S}chr{\"o}dinger operators with singular potential\/}}, J. Differential
  Equations {\bf 265} (2018), pp. 3400--3440.

\bibitem[LP20]{lucia2020addendum}
M.~Lucia and S.~Prashanth,
  \href{http://dx.doi.org/10.1016/j.jde.2020.05.031}{{\em Addendum to
  ``{C}riticality theory for {S}chr{\"o}dinger operators with singular
  potential'' [{J}. {D}iffer. {E}qu. 265 (2018) 3400--3440]\/}}, J.
  Differential Equations {\bf 269} (2020), pp. 7211--7213.

\bibitem[Mur86]{murata1986structure}
M.~Murata, \href{http://dx.doi.org/10.1215/S0012-7094-86-05347-0}{{\em
  Structure of positive solutions to ($-{\Delta}+{V}) u= 0$ in ${\R}^n$\/}},
  Duke Math. J. {\bf 53} (1986), pp. 869--943.

\bibitem[Nag86]{nagy1986operators}
B.~Nagy, \href{http://www.jstor.org/stable/24714364}{{\em Operators with
  spectral singularities\/}}, Journal of Operator Theory  (1986), pp. 307--325.

\bibitem[Nai54]{naimark1954investigation}
M.~A. Naimark, \href{http://mi.mathnet.ru/mmo30}{{\em Investigation of the
  spectrum and the expansion in eigenfunctions of a nonselfadjoint operator of
  the second order on a semi-axis\/}}, Tr. Mosk. Mat. Obs. {\bf 3} (1954), pp.
  181--270.

\bibitem[Pav66]{pavlov1966nonselfadjoint}
B.~S. Pavlov, {\em On a non-selfadjoint {S}chr\"odinger operator\/}, in {\em
  Problems of {M}athematical {P}hysics, {N}o. 1, {S}pectral {T}heory,
  {D}iffraction {P}roblems ({R}ussian)\/}, pp. 102--132, Izdat. Leningrad.
  Univ., Leningrad, 1966.

\bibitem[Pin88]{pinchover1988positive}
Y.~Pinchover, \href{http://dx.doi.org/10.1215/S0012-7094-88-05743-2}{{\em On
  positive solutions of second-order elliptic equations, stability results, and
  classification\/}}, Duke Math. J. {\bf 57} (1988), pp. 955--980.

\bibitem[Pin90]{pinchover1990criticality}
Y.~Pinchover, \href{http://dx.doi.org/10.1016/0022-0396(90)90007-C}{{\em On
  criticality and ground states of second order elliptic equations, {II}\/}},
  J. Differential Equations {\bf 87} (1990), pp. 353--364.

\bibitem[Pin92]{pinchover1992large}
Y.~Pinchover, \href{http://dx.doi.org/10.1016/0022-1236(92)90090-6}{{\em Large
  time behavior of the heat kernel and the behavior of the {G}reen function
  near criticality for nonsymmetric elliptic operators\/}}, J. Funct. Anal.
  {\bf 104} (1992), pp. 54--70.

\bibitem[Pin04]{pinchover2004large}
Y.~Pinchover, \href{http://dx.doi.org/10.1016/S0022-1236(03)00110-1}{{\em Large
  time behavior of the heat kernel\/}}, J. Funct. Anal. {\bf 206} (2004), pp.
  191--209.

\bibitem[Pov50]{povzner1950method}
A.~Y. Povzner,
  \href{http://dspace.univer.kharkov.ua/handle/123456789/1497}{{\em On {M}.{G}.
  {K}rein's method of directing functionals\/}}, Zapiski Inst. Mat. Meh. Harkov
  Gos. Univ. and Harkov. Mat. Obs. {\bf 28} (1950), pp. 43--52.

\bibitem[Pov53]{povzner1953expansion}
A.~Y. Povzner, \href{http://mi.mathnet.ru/eng/msb/v74/i1/p109}{{\em On the
  expansion of arbitrary functions in characteristic functions of the operator
  {$-\Delta u+cu$}\/}}, Mat. Sb. (N.S.) {\bf 32(74)} (1953), pp. 109--156.

\bibitem[PT06]{pinchover2006ground}
Y.~Pinchover and K.~Tintarev,
  \href{http://dx.doi.org/10.1016/j.jfa.2005.05.015}{{\em A ground state
  alternative for singular {S}chr{\"o}dinger operators\/}}, J. Funct. Anal.
  {\bf 230} (2006), pp. 65--77.

\bibitem[PT07]{pinchover2007ground}
Y.~Pinchover and K.~Tintarev,
  \href{http://dx.doi.org/10.1007/s00526-006-0040-2}{{\em Ground state
  alternative for $p$-{L}aplacian with potential term\/}}, Calc. Var. Partial
  Differential Equations {\bf 28} (2007), pp. 179--201.

\bibitem[Rau78]{MR0495958}
J.~Rauch, \href{http://projecteuclid.org/euclid.cmp/1103904212}{{\em Local
  decay of scattering solutions to {S}chr{\"o}dinger's equation\/}}, Comm.
  Math. Phys. {\bf 61} (1978), pp. 149--168.

\bibitem[Rej69]{rejto1969partly}
P.~A. Rejto, \href{http://dx.doi.org/10.1016/0022-247X(69)90065-1}{{\em On
  partly gentle perturbations. {III}\/}}, J. Math. Anal. Appl. {\bf 27} (1969),
  pp. 21--67.

\bibitem[Sch60a]{schwartz1960some}
J.~Schwartz,
  \href{https://onlinelibrary.wiley.com/doi/abs/10.1002/cpa.3160130405}{{\em
  Some non-selfadjoint operators\/}}, Comm. Pure Appl. Math. {\bf 13} (1960),
  pp. 609--639.

\bibitem[Sch60b]{schwinger1960field}
J.~Schwinger,
  \href{http://www.sciencedirect.com/science/article/pii/0003491660900270}{{\em
  Field theory of unstable particles\/}}, Annals of Physics {\bf 9} (1960), pp.
  169--193.

\bibitem[Sim73]{MR0353896}
B.~Simon, \href{http://dx.doi.org/10.2307/1970847}{{\em Resonances in
  {$n$}-body quantum systems with dilatation analytic potentials and the
  foundations of time-dependent perturbation theory\/}}, Ann. of Math. (2) {\bf
  97} (1973), pp. 247--274.

\bibitem[Sim76]{simon1976bound}
B.~Simon, \href{http://math.caltech.edu/SimonPapers/70.pdf}{{\em The bound
  state of weakly coupled {S}chr{\"o}dinger operators in one and two
  dimensions\/}}, Ann. Physics {\bf 97} (1976), pp. 279--288.

\bibitem[Sim81]{simon1981large}
B.~Simon, \href{http://dx.doi.org/10.1016/0022-1236(81)90073-2}{{\em Large time
  behavior of the $l^p$ norm of {S}chr{\"o}dinger semigroups\/}}, Journal of
  Functional analysis {\bf 40} (1981), pp. 66--83.

\bibitem[Smi41]{smirnov1941course}
V.~Smirnov, \href{http://books.e-heritage.ru/book/10077790}{{\em Course of
  higher mathematics\/}}, vol.~4, OGIZ, Leningrad, Moscow, 1941, 1 edn.

\bibitem[Smi64]{smirnov1964course}
V.~I. Smirnov, {\em A {C}ourse of {H}igher {M}athematics: {V}ol. 4, {I}ntegral
  {E}quations and {P}artial {D}ifferential {E}quations\/}, Pergamon Press,
  1964.

\bibitem[Som12]{sommerfeld1912greensche}
A.~Sommerfeld, \href{http://eudml.org/doc/145344}{{\em Die {G}reensche
  {F}unktion der {S}chwingungslgleichung\/}}, Jahresbericht der Deutschen
  Mathematiker-Vereinigung {\bf 21} (1912), pp. 309--353.

\bibitem[Som48]{sommerfeld1948partielle}
A.~J.~W. Sommerfeld, {\em Partielle {D}ifferetialgleichungen der {P}hysik\/},
  Akademische Verlagsgesellschaft Geest \& Portig, Leipzig, 1948, 2 edn.

\bibitem[Sve50]{sveshnikov1950radiation}
A.~Sveshnikov, \href{http://mi.mathnet.ru/eng/dan50544}{{\em Radiation
  principle\/}}, Dokl. Akad. Nauk {\bf 73} (1950), pp. 917--920.

\bibitem[Tit46]{titchmarsh1946eigenfunction}
E.~Titchmarsh, {\em Eigenfunction expansions associated with second-order
  differential equations\/}, Clarendon Press, Oxford, 1946.

\bibitem[TS48]{tikhonov1948radiation}
A.~Tikhonov and A.~Samarskii,
  \href{http://samarskii.ru/articles/1948/1948_003ocr.pdf}{{\em On the
  radiation principle\/}}, Zh. Eksper. Teoret. Fiz. {\bf 18} (1948), pp.
  243--248.

\bibitem[TS51]{tikhonov1951equations}
A.~Tikhonov and A.~Samarskii, {\em Equations of Mathematical Physics\/},
  Gostekhizdat, Moscow, 1951, 1 edn.

\bibitem[TT08]{takac2008generalized}
P.~Tak{\'a}{\v{c}} and K.~Tintarev,
  \href{http://dx.doi.org/10.1017/S0308210506000904}{{\em Generalized minimizer
  solutions for equations with the $p$-{L}aplacian and a potential term\/}},
  Proceedings of the Royal Society of Edinburgh Section A: Mathematics {\bf
  138} (2008), pp. 201--221.

\bibitem[Vai66]{vainberg1966principles}
B.~Vainberg, \href{http://mi.mathnet.ru/rus/umn/v21/i3/p115}{{\em Principles of
  radiation, limit absorption and limit amplitude in the general theory of
  partial differential equations\/}}, Russian Mathematical Surveys {\bf 21}
  (1966), pp. 115--193.

\bibitem[Vai68]{vainberg1968analytical}
B.~R. Vainberg, \href{http://mi.mathnet.ru/eng/msb4058}{{\em On the analytical
  properties of the resolvent for a certain class of operator-pencils\/}}, Mat.
  Sb. (N.S.) {\bf 119} (1968), pp. 259--296.

\bibitem[Vai75]{vainberg1975short}
B.~R. Vainberg, \href{http://mi.mathnet.ru/eng/umn3983}{{\em On the short wave
  asymptotic behaviour of solutions of stationary problems and the asymptotic
  behaviour as $t\to\infty$ of solutions of non-stationary problems\/}},
  Russian Mathematical Surveys {\bf 30} (1975), p.~1.

\bibitem[VG87]{vizgin1987reception}
V.~P. Vizgin and G.~E. Gorelik,
  \href{http://dx.doi.org/10.1007/978-94-009-3875-5_8}{{\em The reception of
  the {Th}eory of {R}elativity in {R}ussia and the {USSR}\/}},
  \href{http://dx.doi.org/10.1007/978-94-009-3875-5_8}{in
  \href{http://dx.doi.org/10.1007/978-94-009-3875-5_8}{{\em The comparative
  reception of relativity\/}}}, pp. 265--326, Springer, 1987.

\bibitem[Vis52]{vishik1952general}
M.~I. Vishik, \href{http://mi.mathnet.ru/eng/mmo/v1/p187}{{\em On general
  boundary problems for elliptic differential equations\/}}, Tr. Mosk. Mat.
  Obs. {\bf 1} (1952), pp. 187--246.

\bibitem[Wei99]{weidl1999remarks}
T.~Weidl, \href{http://dx.doi.org/10.1080/03605309908821417}{{\em Remarks on
  virtual bound states for semi-bounded operators\/}}, Comm. Partial
  Differential Equations {\bf 24} (1999), pp. 25--60.

\bibitem[Wey10]{weyl1910gewohnliche}
H.~Weyl, \href{http://dx.doi.org/10.1007/BF01474161}{{\em \"{U}ber
  gew{\"o}hnliche {D}ifferentialgleichungen mit {S}ingularit{\"a}ten und die
  zugeh{\"o}rigen {E}ntwicklungen willk{\"u}rlicher {F}unktionen\/}},
  Mathematische Annalen {\bf 68} (1910), pp. 220--269.

\bibitem[Wig33]{wigner1933streuung}
E.~P. Wigner, \href{http://dx.doi.org/10.1007/BF01331145}{{\em \"{U}ber die
  {S}treuung von {N}eutronen an {P}rotonen\/}}, Zeitschrift f{\"u}r Physik {\bf
  83} (1933), pp. 253--258.

\bibitem[Yaf74]{yafaev1974theory}
D.~R. Yafaev, \href{http://mi.mathnet.ru/eng/msb/v136/i4/p567}{{\em On the
  theory of the discrete spectrum of the three-particle {S}chr\"{o}dinger
  operator\/}}, Mat. Sb. (N.S.) {\bf 23} (1974), pp. 535--559.

\bibitem[Yaf75]{yafaev1975virtual}
D.~R. Yafaev, \href{http://mi.mathnet.ru/znsl2822}{{\em The virtual level of
  the {S}chr{\"o}dinger equation\/}}, \href{http://mi.mathnet.ru/znsl2822}{in
  \href{http://mi.mathnet.ru/znsl2822}{{\em Mathematical questions in the
  theory of wave propagation, 7\/}}}, vol.~51, pp. 203--216, Nauka, St.
  Petersburg, 1975.

\bibitem[Yaf83]{yafaev1983scattering}
D.~R. Yafaev, \href{http://dx.doi.org/10.1070/sm1983v046n02abeh002785}{{\em
  Scattering subspaces and asymptotic completeness for the time-dependent
  {S}chr{\"o}dinger equation\/}}, Mathematics of the USSR--Sbornik {\bf 46}
  (1983), pp. 267--283.

\bibitem[Yaf10]{MR2598115}
D.~R. Yafaev, \href{http://dx.doi.org/10.1090/surv/158}{{\em Mathematical
  scattering theory\/}}, vol. 158 of {\em Mathematical Surveys and
  Monographs\/}, American Mathematical Society, Providence, RI, 2010.

\bibitem[Yaj05]{yajima2005dispersive}
K.~Yajima, \href{http://dx.doi.org/10.1007/s00220-005-1375-9}{{\em Dispersive
  estimates for {S}chr{\"o}dinger equations with threshold resonance and
  eigenvalue\/}}, Comm. Math. Phys. {\bf 259} (2005), pp. 475--509.

\bibitem[Yam73]{MR0320547}
O.~Yamada, \href{http://dx.doi.org/10.2977/prims/1195192961}{{\em On the
  principle of limiting absorption for the {D}irac operator\/}}, Publ. Res.
  Inst. Math. Sci. {\bf 8} (1972/73), pp. 557--577.

\end{thebibliography}
\end{document}